\documentclass[a4paper,10pt, reqno]{amsart}
\usepackage{amssymb}
\usepackage{amsmath}
\usepackage{amsthm}
\usepackage{bm}
\usepackage{cases}
\usepackage{mathrsfs}
\usepackage{enumerate}
\makeatletter
	
	\@addtoreset{equation}{section}
\makeatother
\theoremstyle{definition}\newtheorem{Def}{Definition}[section]
\theoremstyle{plain}\newtheorem{Thm}{Theorem}[section]
\theoremstyle{plain}\newtheorem{Prop}{Proposition}[section]
\theoremstyle{plain}\newtheorem{Lem}{Lemma}[section]
\theoremstyle{remark}\newtheorem{Rem}{Remark}[section]
\newenvironment{Prob}[1]{

\noindent{\bf \underline{Problem #1.}}\hspace{2mm}}{

}
\title[Navier-Stokes equations under frictional boundary conditions]{On a strong solution of the non-stationary Navier-Stokes equations under slip or leak boundary conditions of friction type}
\author{Takahito Kashiwabara}
\address{Graduate School of Mathematical Sciences, The university of Tokyo, 3-8-1 Komaba, Meguro, Tokyo 153-8914, Japan}
\email{tkashiwa@ms.u-tokyo.ac.jp}
\begin{document}
\begin{abstract}
	Strong solutions of the non-stationary Navier-Stokes equations under non-linearized slip or leak boundary conditions are investigated.
	We show that the problems are formulated by a variational inequality of parabolic type, to which uniqueness is established.
	Using Galerkin's method and deriving a priori estimates, we prove global and local existence for 2D and 3D slip problems respectively.
	For leak problems, under no-leak assumption at $t=0$ we prove local existence in 2D and 3D cases.
	Compatibility conditions for initial states play a significant role in the estimates.
\end{abstract}
\maketitle
\section{Introduction} \label{Sec a}
Let $\Omega$ be a bounded smooth domain in $\mathbb R^d\, (d=2,3)$, and fix $T>0$.
We suppose that the boundary $\Gamma = \partial\Omega$ consists of two nonempty open subsets, that is, $\Gamma = \overline\Gamma_0 \cup \overline\Gamma_1,\, \overline\Gamma_0 \cap \overline\Gamma_1 = \emptyset$.
We are concerned with the non-stationary incompressible Navier-Stokes equations in $\Omega$:
\begin{numcases}{}
	u' + (u\cdot\nabla)u -\nu\Delta u + \nabla p = f & in\quad $\Omega\times(0,T)$, \label{a.3} \\
	\mathrm{div}\, u = 0 & in\quad $\Omega\times(0,T)$,
\end{numcases}
with the initial condition
\begin{equation}
	u = u_0 \qquad \text{in}\quad \Omega\times\{0\}.
\end{equation}
Here, $\nu$, $u$, $p$, and $f$ denote a viscosity constant, velocity field, pressure, and external force respectively; $u'$ means the time derivative $\frac{\partial u}{\partial t}$.

As for the boundary condition, we impose the adhesive b.c.~on $\Gamma_0$:
\begin{equation}
	u=0 \quad\text{on}\quad \Gamma_0. \label{a.5}
\end{equation}
On the other hand, we consider one of the following nonlinear b.c.~on $\Gamma_1$:
\begin{equation}
	u_n = 0, \qquad |\sigma_\tau|\le g, \qquad \sigma_\tau\cdot u_\tau + g|u_\tau| = 0, \quad\text{on}\quad \Gamma_1, \label{a.1}
\end{equation}
which is called the \emph{slip boundary condition of friction type} (SBCF), and
\begin{equation}
	u_\tau = 0, \qquad |\sigma_n|\le g, \qquad \sigma_n u_n + g|u_n| = 0, \quad\text{on}\quad \Gamma_1, \label{a.2}
\end{equation}
which is called the \emph{leak boundary condition of friction type} (LBCF).
Here, $n$ is an outer unit normal vector defined on $\Gamma$, and we write $u_n:=u\cdot n$ and $u_\tau:=u-u_nn$.
The stress tensor $\mathbb T = (T_{ij})_{i,j=1,. . . ,d}$ is given by $T_{ij} = -p\delta_{ij} + \nu ( \frac{\partial u_i}{\partial x_j} + \frac{\partial u_j}{\partial x_i} )$, $\delta_{ij}$ being Kronecker delta.
We define the stress vector $\sigma = \sigma(u,p)$ as $\sigma = \mathbb Tn$, and write $\sigma_n:=\sigma\cdot n$ and $\sigma_\tau:=\sigma-\sigma_nn$.
One can easily see that $\sigma_n=\sigma_n(u,p)$ may depend on $p$, whereas $\sigma_\tau=\sigma_\tau(u)$ does not.

The function $g$, given on $\Gamma_1$ and assumed to be strictly positive, is called a \emph{modulus of friction}.
Its physical meaning is the threshold of the tangential (resp.~normal) stress.
In fact, if $|\sigma_\tau|<g$ (resp.~$|\sigma_n|<g$) then (\ref{a.1}) (resp.~(\ref{a.2})) implies $u_\tau=0$ (resp.~$u_n=0$), namely, no slip (resp.~leak) occurs;
otherwise non-trivial slip (resp.~leak) can take place.
We notice that if we make $g=0$ formally, (\ref{a.1}) and (\ref{a.2}) reduce to the usual slip and leak b.c.~respectively.
In summary, SBCF and LBCF are non-linearized slip and leak b.c.~obtained from introduction of some friction law on the stress.

It should be also noted that the second and third conditions of (\ref{a.1}) (resp.~(\ref{a.2})) are equivalently rewritten, with the notation of subdifferential, as
\begin{equation*}
	\sigma_\tau \in -g\partial |u_\tau| \qquad (\text{resp.}\quad \sigma_n \in -g\partial |u_n|).
\end{equation*}
Though we will not pursue this matter further, one can refer to \cite{Che03, Kol00} for the Navier-Stokes equations with general subdifferential b.c.
See also \cite{Con06}, which considers the motion of a Bingham fluid under b.c.~with nonlocal friction against slip.

SBCF and LBCF are first introduced in \cite{F94,FKS95} for the stationary Stokes and Navier-Stokes equations, where existence and uniqueness of weak solutions are established.
Generalized SBCF is considered in \cite{R05,RT07}.
The $H^2$-$H^1$ regularity for the Stokes equations is proved in \cite{S04}.
In terms of numerical analysis, \cite{AGS10, Kas11_slip, Kas11_leak, LiAn11, LiLi08, LiLi10_1, LiLi10_2} deal with finite element methods for SBCF or LBCF.
Applications of SBCF and LBCF to realistic problems, together with numerical simulations, are found in \cite{KFS98, SK04}.

For non-stationary cases, \cite{Fuj01, Fuj02} study the time-dependent Stokes equations without external forces under SBCF and LBCF, using a nonlinear semigroup theory.
The solvability of nonlinear problems are discussed in \cite{LiAn11_2} for SBCF, and in \cite{ALL09} for a variant of LBCF.
They use the Stokes operator associated with the linear slip or leak b.c., and do not take into account a compatibility condition at $t=0$.

The purpose of this paper is to prove existence and uniqueness of a strong solution for (\ref{a.3})--(\ref{a.5}) with (\ref{a.1}) or (\ref{a.2}).
We employ the class of solutions of Ladyzhenskaya type (see \cite{Lad69}), searching $(u,p)$ such that
\begin{equation*}
\begin{cases}
	u\in L^\infty(0, T; H^1(\Omega)^d),\quad u' \in L^\infty(0, T; L^2(\Omega)^d) \cap L^2(0, T; H^1(\Omega)^d), \\
	p\in L^\infty(0, T; L^2(\Omega)).
\end{cases}
\end{equation*}

There are several reasons we focus on this strong solution.
First, from a viewpoint of numerical analysis, we would like to construct solutions in a class where uniqueness and regularity are assured also for 3D case.
Second, we desire an $L^\infty$-estimate with respect to time for $p$, which may not be obtained for weak solutions of Leray-Hopf type (cf.~\cite[Proposition III.1.1]{Tem77}).
Third, in LBCF, it is not straightforward to deduce a weak solution because of (\ref{a.4}) below.
Similar difficulty already comes up in the linear leak b.c.~(see \cite{Mar07})

The rest of this paper is organized as follows. 
Basic symbols, notation, and function spaces are given in Section 2.

In Section 3, we investigate the problem with SBCF.
The weak formulation is given by a variational inequality, to which we prove uniqueness of solutions.
To show existence, we consider a regularized problem, approximate it by Galerkin's method, and derive a priori estimates which allow us to pass on the limit to deduce the desired strong solution.
Using the compatibility condition that $u_0$ must satisfy SBCF, we can adapt $u_0$ to the regularized problem, which makes an essential point in the estimate.

Section 4 is devoted to a study of the problem with LBCF. There are two major differences from SBCF.
First, as was pointed out in the stationary case \cite[Remark 3.2]{F94}, we cannot obtain the uniqueness of an additive constant for $p$ if no leak occurs, namely, $u_n = 0$ on $\Gamma_1$.
Second, under LBCF, the quantity
\begin{equation}
	\int_\Omega \Big\{ (u\cdot\nabla) v\cdot v \Big\} dx = \frac12 \int_\Gamma u_n|v|^2\, ds \qquad (\text{if}\quad \mathrm{div}\,u = 0) \label{a.4}
\end{equation}
need not vanish because $u_n$ can be non-zero.
This fact affects our a priori estimates badly, and we can extract a solution only when the initial leak $\|u_{0n}\|_{L^2(\Gamma_1)}$ is small enough.
Incidentally, if we use the so-called Bernoulli pressure $p + \frac12 |u|^2$ instead of standard $p$, the mathematical difficulty arising from (\ref{a.4}) are resolved;
nevertheless the leak b.c.~involving the Bernoulli pressure is known to cause an unphysical effect in numerical simulations (see \cite[p.338]{HRT96}).
Thereby we employ the usual formulation.

Finally, in Section 5 we conclude this paper with some remarks on higher regularity.

\section{Preliminaries}
Throughout the present paper, the domain $\Omega$ is supposed to be as smooth as required.
For the precise regularity of $\Omega$ which is sufficient to deduce our main theorems, see Remarks \ref{Rem c.10} and \ref{Rem d.10}.
We shall denote by $C$ various generic positive constants depending only on $\Omega$, unless otherwise stated.
When we need to specify dependence on a particular parameter, we write as $C = C(f, g, u_0)$, and so on.

We use the Lebesgue space $L^p(\Omega)\, (1\le p\le \infty)$, and the Sobolev space $H^r(\Omega) = \{ \phi \!\in\! L^2(\Omega) \,|\, \|\phi\|_{H^r(\Omega)}^2 \!=\! \sum_{|\alpha|\le r} \|\partial^\alpha \phi\|_{L^2(\Omega)}^2 \!<\! \infty\}$ for a nonnegative integer $r$, where $H^0(\Omega)$ means $L^2(\Omega)$.
$H^s(\Omega)$ is also defined for a non-integer $s>0$ (e.g.~\cite[Definition 1.2]{GiRa86}).
We put $L^2_0(\Omega) = \{ q\in L^2(\Omega) \,\big|\, \int_\Omega q\, dx = 0 \}$.
For spaces of vector-valued functions, we write $L^p(\Omega)^d$, and so on.

The Lebesgue and Sobolev spaces on the boundary $\Gamma$, $\Gamma_0$, or $\Gamma_1$, are also used.
$H^0(\Gamma_1)$ means $L^2(\Gamma_1)$, and we put $L^2_0(\Gamma_1) = \{ \eta \in L^2(\Gamma_1) \,\big|\, \int_{\Gamma_1} \eta\, ds = 0 \}$, where $ds$ denotes the surface measure.
For a positive function $g$ on $\Gamma_1$, the weighted Lebesgue spaces $L^1_g(\Gamma_1)$ and $L^\infty_{1/g}(\Gamma_1)$ are defined by the norms
\begin{equation*}
	\|\eta\|_{L^1_g(\Gamma_1)} = \int_{\Gamma_1}g|\eta|\,ds \qquad\text{and}\qquad \|\eta\|_{L^\infty_{1/g}(\Gamma_1)} = \mathrm{ess.}\sup_{\Gamma_1} \frac{|\eta|}g,
\end{equation*}
respectively. The dual space of $L^1_g(\Gamma_1)$ is $L^\infty_{1/g}(\Gamma_1)$ (see \cite[Lemma 2.1]{F94}).

The usual trace operator $\phi\mapsto\phi|_\Gamma$ is defined from $H^1(\Omega)$ onto $H^{1/2}(\Gamma)$.
The restrictions $\phi|_{\Gamma_0}$, $\phi|_{\Gamma_1}$ of $\phi|_\Gamma$, are also considered, and
we simply write $\phi$ to indicate them when there is no fear of confusion.
In particular, $\eta_n$ and $\eta_\tau$ means $(\eta\cdot n)|_{\Gamma}$ and $(\eta - (\eta\cdot n)n)|_{\Gamma}$ respectively, for $\eta\in H^{1/2}(\Gamma)^d$.
Note that $\|\eta_n\|_{H^{1/2}(\Gamma)} \le C\|\eta\|_{H^{1/2}(\Gamma)^d}$ and $\|\eta_\tau\|_{H^{1/2}(\Gamma)^d} \le C\|\eta\|_{H^{1/2}(\Gamma)^d}$ because $n$ is smooth on $\Gamma$.

The inner product of $L^2(\Omega)^d$ is simplified as $(\cdot, \cdot)$, while other inner products and norms are written with clear subscripts, e.g., $(\cdot, \cdot)_{L^2(\Gamma_1)}$ or $\|\cdot\|_{H^1(\Omega)^d}$.
For a Banach space $X$, we denote its dual space by $X'$ and the dual product between $X'$ and $X$ by $\left< \cdot, \cdot \right>_X$.
Moreover, we employ the standard notation of Bochner spaces such as $L^2(0,T; X)$, $H^1(0,T; X)$.

For function spaces corresponding to a velocity and pressure, we introduce closed subspaces of $H^1(\Omega)^d$ or $L^2(\Omega)$ as follows:
\begin{align*}
V &= \{ v\in H^1(\Omega)^d \,\big|\, v = 0 \text{ on } \Gamma_0 \},&\mathring V &= \{ v\in H^1(\Omega)^d \,\big|\, v = 0 \text{ on } \Gamma \}, \\
	V_n &= \{ v\in V \,\big|\, v_n = 0 \text{ on } \Gamma_1 \},& V_\tau &= \{ v\in V \,\big|\, v_\tau = 0 \text{ on } \Gamma_1 \}, \\
	Q &= L^2(\Omega),&\mathring Q &= L^2_0(\Omega).
\end{align*}
To indicate a divergence-free space, we set $H^1_\sigma(\Omega)^d = \{ v\in H^1(\Omega)^d \,\big|\, \mathrm{div}\, v = 0 \}$.
We use the notation $V_\sigma = V\cap H^1_\sigma(\Omega)^d$, $\mathring V_\sigma = \mathring V\cap H^1_\sigma(\Omega)^d$, $V_{n,\sigma} = V_n\cap H^1_\sigma(\Omega)^d$, and $V_{\tau,\sigma} = V_\tau\cap H^1_\sigma(\Omega)^d$.

Let us define bilinear forms $a_0$, $b$, and a trilinear form $a_1$ by 
{\allowdisplaybreaks
\begin{align*}
	a_0(u,v) &= \frac\nu2 \sum_{i,j=1}^d \int_\Omega \left( \frac{\partial u_i}{\partial u_j} + \frac{\partial u_j}{\partial u_i} \right) \left( \frac{\partial v_i}{\partial x_j} + \frac{\partial v_j}{\partial x_i} \right) \, dx \hspace{-10.5mm}&(u,v \in H^1(\Omega)^d), \\
	a_1(u,v,w) &= \int_\Omega \left\{ (u\cdot\nabla)v \right\} \cdot w\,dx &(u,v,w \in H^1(\Omega)^d), \\
	b(v,q) &= -\int_\Omega \mathrm{div}v\, q\,dx & (v\in H^1(\Omega)^d,\, q\in L^2(\Omega)).
\end{align*}
}The bilinear forms $a_0, b$ are continuous, and from Korn's inequality (\cite[Lemma 6.2]{KO88}) there exists a constant $\alpha>0$ such that
\begin{equation}
	a_0(v,v) \ge \alpha \|v\|_{H^1(\Omega)^d}^2 \qquad (\forall v\in V). \label{b.3}
\end{equation}

Concerning the trilinear term $a_1$, we obtain the following two lemmas.
\begin{Lem}
	{\rm(i)} When $d=2$, for all $u,v,w\in H^1(\Omega)^d$ it holds that
	\begin{equation}
		|a_1(u,v,w)| \le C \|u\|_{L^2(\Omega)^d}^{1/2} \|u\|_{H^1(\Omega)^d}^{1/2} \|v\|_{H^1(\Omega)^d} \|w\|_{L^2(\Omega)^d}^{1/2} \|w\|_{H^1(\Omega)^d}^{1/2}. \label{b.5}
	\end{equation}
	
	{\rm(ii)} When $d=2$ or $d=3$, for all $u,v,w\in H^1(\Omega)^d$ it holds that
	\begin{equation}
		|a_1(u,v,w)| \le C \|u\|_{L^2(\Omega)^d}^{1/4} \|u\|_{H^1(\Omega)^d}^{3/4} \|v\|_{H^1(\Omega)^d} \|w\|_{L^2(\Omega)^d}^{1/4} \|w\|_{H^1(\Omega)^d}^{3/4}. \label{b.6}
	\end{equation}
\end{Lem}
\begin{Rem}
	In particular, we see from (\ref{b.6}) that
	\begin{equation}
		|a_1(u,v,w)| \le C \|u\|_{H^1(\Omega)^d} \|v\|_{H^1(\Omega)^d} \|w\|_{H^1(\Omega)^d}. \label{b.7}
	\end{equation}
\end{Rem}
\begin{proof}
	By the Sobolev embedding $H^{1/2}(\Omega)\subset L^4(\Omega)$ (resp.~$H^{3/4}(\Omega)\subset L^4(\Omega)$) which is valid for $d=2$ (resp.~$d=2,3$), combined with an interpolation inequality between $L^2(\Omega)$ and $H^1(\Omega)$, we have
	\begin{align*}
		&\|u\|_{L^4(\Omega)^d} \le C\|u\|_{H^{1/2}(\Omega)^d} \le C\|u\|_{L^2(\Omega)^d}^{1/2} \|u\|_{H^1(\Omega)^d}^{1/2} & & (d=2). \\
		(\text{resp.}\quad &\|u\|_{L^4(\Omega)^d} \le C\|u\|_{H^{3/4}(\Omega)^d} \le C\|u\|_{L^2(\Omega)^d}^{1/4} \|u\|_{H^1(\Omega)^d}^{3/4} & & (d=2,3).)
	\end{align*}
	Therefore, since $|a_1(u, v, w)| \le C \|u\|_{L^4(\Omega)^d} \|v\|_{H^1(\Omega)^d} \|w\|_{L^4(\Omega)^d}$ by H\"older's inequality,
	we conclude (\ref{b.5}) (resp.~(\ref{b.6})).
\end{proof}

\begin{Lem} \label{Lem b.2}	
	{\rm (i)} For all $u\in V_{n,\sigma}$ and $v\in H^1(\Omega)^d$, $a_1(u,v,v) = 0$.
	
	{\rm (ii)} For all $u\in V_{\tau,\sigma}$ and $v\in H^1(\Omega)^d$, $a_1(u,v,v) = \frac12 \int_{\Gamma_1}u_n|v|^2\,ds$, and
	\begin{equation}
		|a_1(u, v, v)| \le \gamma_1\|u_n\|_{L^2(\Gamma_1)} \|v\|_{H^1(\Omega)^d}^2, \label{b.10}
	\end{equation}
	where $\gamma_1$ is a constant depending only on $\Omega$.
\end{Lem}
\begin{proof}
	By integration by parts, we have
	\begin{equation*}
		a_1(u,v,w) + a_1(u,w,v) = -\int_{\Omega} \mathrm{div}u\, v\cdot w\, dx + \int_\Gamma u_n v\cdot w\,ds,
	\end{equation*}
	from which the conclusion of (i) and the first assertion of (ii) follow.
	Combining the H\"older inequality $|a_1(u,v,v)| \le C \|u_n\|_{L^2(\Gamma_1)} \|v\|_{L^4(\Gamma_1)^d}^2$, the Sobolev embedding $H^{1/2}(\Gamma_1)\subset L^4(\Gamma_1)$ $(d=2,3)$, and the continuity of the trace operator $H^1(\Omega)\to H^{1/2}(\Gamma_1)$, we derive (\ref{b.10}).
\end{proof}
\begin{Rem}
	Whether $\gamma_1$ is small or not, especially when compared to $\alpha$ in (\ref{b.3}), is a very crucial point in our a priori estimates for LBCF (see Proposition \ref{Prop d.1}).
	This is why we distinguish $\gamma_1$ from other constants $C$ and do not combine $\gamma_1$ with them.
	As (i) above shows, this problem does not happen when we consider SBCF.
\end{Rem}

Furthermore, we introduce nonlinear functionals $j_\tau$ and $j_n$ by
\begin{equation*}
	j_\tau(\eta) = \int_{\Gamma_1} g|\eta|\,ds \quad (\eta\in L^2(\Gamma_1)^d) \quad\text{and}\quad j_n(\eta) = \int_{\Gamma_1} g|\eta|\,ds \quad (\eta\in L^2(\Gamma_1)),
\end{equation*}
where $g>0$ is a modulus of friction mentioned in Section \ref{Sec a}.
They are obviously nonnegative and positively homogeneous.
In addition, they are Lipschitz continuous when $g(t)\in L^2(\Gamma_1)$ for a.e.~$t\in(0, T)$.

The followings, which are readily obtainable consequences of standard trace and (solenoidal) extension theorems (\cite[Theorems I.1.5-6, Lemma I.2.2]{GiRa86}, see also \cite[Section 5.3]{KO88}), are frequently used in subsequent arguments.
\begin{Lem}\label{Lem b.1}
	\rm{(i)} For $v\in V_n$, it holds that $\|v_\tau\|_{H^{1/2}(\Gamma_1)^d} \le C\|v\|_{H^1(\Omega)^d}$.
	
	\rm{(ii)} For $\eta\in H^{1/2}(\Gamma_1)^d$ satisfying $\eta_n = 0 \text{ on }\Gamma_1$, there exists $v\in V_{n,\sigma}$ such that $v_\tau = \eta$ on $\Gamma_1$ and $\|v\|_{H^1(\Omega)^d} \le C \|\eta\|_{H^{1/2}(\Gamma_1)^d}$.
\end{Lem}

\begin{Lem} \label{Lem b.10}
	\rm{(i)} For $v\in V_\tau$, it holds that $\|v_n\|_{H^{1/2}(\Gamma_1)} \le C \|v\|_{H^1(\Omega)^d}$.
	
	\rm{(ii)} For $\eta\in H^{1/2}(\Gamma_1)$ (resp.~$\eta\in H^{1/2}(\Gamma_1)\cap L^2_0(\Gamma_1)$), there exists $v\in V_\tau$ (resp.~$v\in V_{\tau,\sigma}$) such that $v_n = \eta$ on $\Gamma_1$ and $\|v\|_{H^1(\Omega)^d} \le C \|\eta\|_{H^{1/2}(\Gamma_1)}$.
\end{Lem}

The definition of  $\sigma(u,p)$ given in Section \ref{Sec a} becomes ambiguous when $(u,p)$ has only lower regularity, say $u\in H^1(\Omega)^d,\, p\in L^2(\Omega)$.
Thus we propose a redefinition of it, based on the following Green formula:
\begin{equation*}
	(-\nu\Delta u + \nabla p, v) + \int_\Gamma \sigma(u,p)\cdot v\, ds = a_0(u,v) + b(v,p) \qquad (\text{if}\quad \mathrm{div}\, u = 0).
\end{equation*}

\begin{Def} \label{Def b.1}
	Let $u(t)\in V_\sigma$, $p(t)\in Q$, $u'(t)\in L^2(\Omega)^d$, $f(t)\in L^2(\Omega)^d$.
	If (\ref{a.3}) holds in the distribution sense for a.e.~$t\in(0,T)$, that is,
	\begin{equation}
		(u',v) + a_0(u,v) + a_1(u,u,v) + b(v,p) = (f, v) \qquad (\forall v\in\mathring V), \label{b.50}
	\end{equation}
	then we define $\sigma = \sigma(u,p) \in (H^{1/2}(\Gamma_1)^d)'$ by
	\begin{equation}
		\left< \sigma, v \right>_{H^{1/2}(\Gamma_1)^d} = a_0(u,v) + b(v,p) - \left<F, v\right>_V \qquad (\forall v\in V), \label{b.1}
	\end{equation}
	where $F(t)\in V'$ is given by $\left< F, v \right>_V = (f,v) - (u',v) - a_1(u,u,v)$.
\end{Def}

The above $\sigma$ is well-defined by virtue of the trace and extension theorem.
It coincides with the previous definition when $(u,p)$ is sufficiently smooth.
In addition, by Lemmas \ref{Lem b.1} and \ref{Lem b.10}, $\sigma_\tau = \sigma - (\sigma\cdot n)n \in (H^{1/2}(\Gamma_1)^d)'$ and $\sigma_n = \sigma\cdot n \in H^{1/2}(\Gamma_1)'$ are characterized by
\begin{equation*}
	\begin{cases}
		\left< \sigma_\tau, \eta n \right>_{H^{1/2}(\Gamma_1)^d} = 0 & (\forall \eta\in H^{1/2}(\Gamma_1)), \\
		\left< \sigma_\tau, v_\tau \right>_{H^{1/2}(\Gamma_1)^d} = a_0(u,v) + b(v,p) - \left< F, v \right>_{V_n} & (\forall v\in V_n),
	\end{cases}
\end{equation*}
and
\begin{equation*}
	\left< \sigma_n, v_n \right>_{H^{1/2}(\Gamma_1)} = a_0(u,v) + b(v,p) - \left< F, v \right>_{V_\tau} \qquad (\forall v\in V_\tau),
\end{equation*}
respectively. By Lemma \ref{Lem b.1}(ii), $\sigma_\tau$ actually does not depend on $p$.

\section{Navier-Stokes Problem with SBCF}
\subsection{Weak formulations}
Throughout this section, we assume $f \!\in\! L^2(\Omega\times (0,T))^d$, $u_0\in V_{n,\sigma}$, and $g\in L^2(\Gamma_1\times(0,T))$.
Further regularity assumptions on these data will be given before Theorem \ref{Thm c.1}.
In addition, the barrier term $j_\tau$ is simply written as $j$.
A primal weak formulation of (\ref{a.3})--(\ref{a.5}) with (\ref{a.1}) is as follows:
\vspace{2mm}
\begin{Prob}{PDE-SBCF}
	For a.e.~$t\in(0,T)$, find $(u(t), p(t))\in V_n\times\mathring Q$ such that $u'(t) \in L^2(\Omega)^d$, $u(0) = u_0$, $\sigma_\tau$ is well-defined in the sense of Definition \ref{Def b.1}, $|\sigma_\tau| \le g$ a.e.~on $\Gamma_1$, and $\sigma_\tau\cdot u_\tau + g|u_\tau| = 0$ a.e.~on $\Gamma_1$.
\end{Prob}
\begin{Rem}
	More precisely, ``$|\sigma_\tau| \le g$" implies that $\sigma_\tau \in (H^{1/2}(\Gamma_1)^d)'$ actually belongs to $L^\infty_{1/g}(\Gamma_1)^d$ with $\|\sigma_\tau\|_{L^\infty_{1/g}(\Gamma_1)^d} \le 1$.
	In particular, $\sigma_\tau \in L^2(\Gamma_1)^d$.
\end{Rem}
Throughout this section, we refer to Problem PDE-SBCF just as Problem PDE.
Similar abbreviation will be made for other problems.

One can easily find that a classical solution of (\ref{a.3})--(\ref{a.5}) with (\ref{a.1}) solves Problem PDE, and that a sufficiently smooth solution of Problem PDE is a classical solution.
As the next theorem shows, Problem PDE is equivalent to the following variational inequality problem.
\vspace{2mm}
\begin{Prob}{VI$_\sigma$-SBCF}
	For a.e.~$t\in(0,T)$, find $u(t)\in V_{n,\sigma}$  such that $u'(t) \in L^2(\Omega)^d$, $u(0) = u_0$, and 
	\begin{equation}
		(u', v-u) + a_0(u, v-u) + a_1(u, u, v-u) + j(v_\tau) - j(u_\tau) \ge (f, v-u) \quad (\forall v\in V_{n,\sigma}). \label{c.4}
	\end{equation}
\end{Prob}

\begin{Thm}\label{Thm c.2}
	Problems $\mathrm{PDE}$ and $\mathrm{VI}_\sigma$ are equivalent.
\end{Thm}
\begin{Rem}
	The precise meaning of ``equivalent" is that if $(u,p)$ solves Problem PDE, $u$ solves Problem VI$_\sigma$;
	if $u$ solves Problem VI$_\sigma$, there exists unique $p$ such that $(u,p)$ solves Problem PDE.
	Hereafter we will frequently use the terminology ``equivalent" in a similar sense.
\end{Rem}
\begin{proof}
	Let $(u, p)$ be a solution of Problem PDE. Then it follows that
	\begin{equation}
		(u',v) + a_0(u,v) + a_1(u,u,v) + b(v,p) - (\sigma_\tau,v_\tau)_{L^2(\Gamma_1)^d} = (f,v) \quad (\forall v\in V_n). \label{c.2}
	\end{equation}
	Using this equation together with $|\sigma_\tau| \le g$ and $\sigma_\tau\cdot u_\tau + g|u_\tau| = 0$, we have
	\begin{align*}
		&\quad (u', v-u) + a_0(u, v-u) + a_1(u, u, v-u) + j(v_\tau) - j(u_\tau) - (f, v-u) \\
		&=-(\sigma_\tau, v_\tau - u_\tau)_{L^2(\Gamma_1)^d} + j(v_\tau) - j(u_\tau) = \int_{\Gamma_1} (g|v_\tau| - \sigma_\tau v_\tau)ds \ge 0,
	\end{align*}
	for all $v\in V_{n,\sigma}$. Hence $u$ is a solution of Problem VI$_\sigma$.
	
	Next, let $u$ be a solution of Problem VI$_\sigma$. Taking $u\pm v$ as a test function in (\ref{c.4}), with arbitrary $v\in\mathring V_\sigma$, we find that
	\begin{equation}
		(u', v) + a_0(u, v) + a_1(u, u, v) = (f, v) \qquad (\forall v\in\mathring V_\sigma). \label{c.33}
	\end{equation}
	By a standard theory (see \cite[Propositions I.1.1 and I.1.2]{Tem77}), there exists unique $p\in\mathring Q$ such that (\ref{b.50}) holds.
	Therefore, $\sigma_\tau \in (H^{1/2}(\Gamma_1)^d)'$ is well-defined, and thus
	\begin{equation*}
		(u', v) + a_0(u, v) + a_1(u, u, v) + b(v, p) - \left< \sigma_\tau, v_\tau \right>_{H^{1/2}(\Gamma_1)^d} = (f, v) \qquad (\forall v\in V_n).
	\end{equation*}
	Combining this equation with (\ref{c.4}), we obtain
	\begin{equation}
		-\left< \sigma_\tau, v_\tau - u_\tau \right>_{H^{1/2}(\Gamma_1)^d} \le \int_{\Gamma_1} g(|v_\tau| - |u_\tau|)ds \qquad (\forall v\in V_{n,\sigma}), \label{c.5}
	\end{equation}
	and as a result of triangle inequality, $|\left< \sigma_\tau, v_\tau \right>_{H^{1/2}(\Gamma_1)^d}| \le \int_{\Gamma_1}g|v_\tau|\,ds$ for $v\in V_{n,\sigma}$.
	In view of Lemma \ref{Lem b.1}(ii), this implies that for $\eta\in H^{1/2}(\Gamma_1)^d$
	\begin{equation*}
		|\left< \sigma_\tau, \eta \right>_{H^{1/2}(\Gamma_1)^d}| = |\left< \sigma_\tau, \eta_\tau \right>_{H^{1/2}(\Gamma_1)^d}| \le \|\eta_\tau\|_{L^1_{g}(\Gamma_1)^d} \le \|\eta\|_{L^1_{g}(\Gamma_1)^d}.
	\end{equation*}
	By a density argument, we can extend $\sigma_\tau$ to an element of $(L^1_g(\Gamma)^d)'$ such that
	\begin{equation*}
		|\left< \sigma_\tau, \eta \right>_{L^1_g(\Gamma_1)^d}| \le \|\eta\|_{L^1_g(\Gamma_1)^d} \qquad (\forall \eta\in L^1_g(\Gamma_1)^d).
	\end{equation*}
	Since $(L^1_g(\Gamma_1)^d)' = L^\infty_{1/g}(\Gamma_1)^d$, we conclude $|\sigma_\tau| \le g$.
	Then $\sigma_\tau\cdot u_\tau + g|u_\tau| = 0$ follows from (\ref{c.5}) with $v=0$.
	Hence $(u,p)$ is a solution of Problem PDE.
\end{proof}

\subsection{Main theorem. Proof of uniqueness.}
We are now in a position to state our main theorem. We assume:
\begin{enumerate}[\qquad(S1)]
	\item $f\in H^1(0,T; L^2(\Omega)^d)$. \label{S1}
	\item $g\in H^1(0, T; L^2(\Gamma_1))$ with $g(0)\in H^1(\Gamma_1)$. \label{S2}
	\item $u_0\in H^2(\Omega)^d\cap V_{n,\sigma}$, and SBCF is satisfied at $t=0$, namely, \label{S3}
		\begin{equation*}
			|\sigma_\tau(u_0)| \le g(0) \quad\text{and}\quad \sigma_\tau(u_0) u_{0\tau} + g(0)|u_{0\tau}| = 0 \quad\text{a.e.~on }\; \Gamma_1.
		\end{equation*}
\end{enumerate}
Note that $\sigma_\tau(u_0)$ can be defined in a usual sense.

\begin{Thm}\label{Thm c.1}
	Under $\mathrm{(S\ref{S1})}$--$\mathrm{(S\ref{S3})}$, when $d=2$ there exists a unique solution $u$ of Problem $\mathrm{VI}_\sigma$ such that
	\begin{equation*}
		u\in L^\infty(0,T; V_{n,\sigma}), \qquad u'\in L^\infty(0,T; L^2(\Omega)^d)\cap L^2(0,T; V_{n,\sigma}).
	\end{equation*}
	When $d=3$, the same conclusion holds on some smaller time interval $(0,T')$.
\end{Thm}

We call the solution in the above theorem a \emph{strong solution} of Problem VI$_\sigma$.
First we prove the uniqueness of a strong solution.
The existence will be proved in Section \ref{Sec 3.4} after some additional preparations.
\begin{Prop} \label{Prop c.4}
	If $u_1$ and $u_2$ are strong solutions of Problem $\mathrm{VI}_\sigma$, then $u_1 = u_2$.
\end{Prop}
\begin{proof}
	Taking $v=u_2$ and $v=u_1$ in (\ref{c.4}) for $u_1$ and that for $u_2$ respectively, and adding the resulting two inequalities, for a.e.~$t\in (0,T)$ we obtain
	\begin{align}
		&\quad (u_1'-u_2', u_1-u_2) + a_0(u_1-u_2, u_1-u_2) \notag \\
		&\le a_1(u_1, u_1, u_2 - u_1) + a_1(u_2, u_2, u_1 - u_2) \notag \\
		&= -a_1(u_1-u_2, u_2, u_1-u_2) - a_1(u_2, u_1-u_2, u_1-u_2). \label{c.12}
	\end{align}
	We deduce from (\ref{b.6}), together with Young's inequality, that
	\begin{align*}
		|a_1(u_1-u_2, u_2, u_1-u_2)| &\le C \|u_1-u_2\|_{L^2(\Omega)^d}^{1/2} \|u_1-u_2\|_{H^1(\Omega)^d}^{3/2} \|u_2\|_{H^1(\Omega)^d} \\
			&\le \frac\alpha2 \|u_1- u_2\|_{H^1(\Omega)^d}^2 + C\|u_2\|_{H^1(\Omega)^d}^2 \|u_1- u_2\|_{L^2(\Omega)^d}^2,
	\end{align*}
	\vspace{-5mm}
	\begin{align*}
		|a_1(u_2, u_1-u_2, u_1-u_2)| &\le C \|u_2\|_{H^1(\Omega)^d} \|u_1 - u_2\|_{H^1(\Omega)^d}^{7/4} \|u_1 - u_2\|_{L^2(\Omega)^d}^{1/4} \\
			&\le \frac\alpha2 \|u_1- u_2\|_{H^1(\Omega)^d}^2 + C\|u_2\|_{H^1(\Omega)^d}^8 \|u_1- u_2\|_{L^2(\Omega)^d}^2.
	\end{align*}
	Combining (\ref{b.3}) and these estimates with (\ref{c.12}), we have
	\begin{equation*}
		\frac d{dt} \|u_1 - u_2\|_{L^2(\Omega)^d}^2 \le C(\|u_2\|_{H^1(\Omega)^d}^2 + \|u_2\|_{H^1(\Omega)^d}^8) \|u_1 - u_2\|_{L^2(\Omega)^d}^2.
	\end{equation*}
	By Gronwall's inequality, we conclude
	\begin{equation*}
		\|u_1(t) - u_2(t)\|_{L^2(\Omega)^d}^2 \le e^{\int_0^t C(\|u_2\|_{H^1(\Omega)^d}^2 + \|u_2\|_{H^1(\Omega)^d}^8)\, dt} \|u_1(0) - u_2(0)\|_{L^2(\Omega)^d}^2 = 0,
	\end{equation*} 
	since $u_1(0) = u_2(0) = u_0$. (Note that $\int_0^t (\|u_2\|_{H^1(\Omega)^d}^2 + \|u_2\|_{H^1(\Omega)^d}^8)\, dt$ remains finite because $u\in L^\infty(0, T; H^1(\Omega)^d.)$
	Thus $u_1(t) = u_2(t)$.
\end{proof}
\begin{Rem}
	In the case of SBCF here, the last term of (\ref{c.12}) vanishes, according to Lemma \ref{Lem b.2}(i).
	We did not use that fact because we would like to make our proof of uniqueness remain unchanged when we deal with LBCF.
\end{Rem}

Concerning the associated pressure, we find:
\begin{Prop}\label{Prop c.5}
	Under the assumptions of Theorem $\ref{Thm c.1}$, let $u$ be the strong solution of Problem $\mathrm{VI}_\sigma$, and $p$ be the associated pressure obtained in the proof of Theorem $\ref{Thm c.2}$.
	Then $p\in L^\infty(0, T;\mathring Q)$.
\end{Prop}
\begin{proof}
	For a.e.~$t\in (0,T)$, the well-known inf-sup condition (see \cite[I.(5.14)]{GiRa86}), together with (\ref{c.2}), (\ref{b.7}), and $|\sigma_\tau| \le g$ a.e.~on $\Gamma_1$, yields
	\begin{align*}
		\|p\|_{L^2(\Omega)} &\le \sup_{v\in\mathring V} \frac{b(v, p)}{\|v\|_{H^1(\Omega)^d}} \\
			&\le \|u'\|_{L^2(\Omega)^d} + C\|u\|_{H^1(\Omega)^d} + C\|u\|_{H^1(\Omega)^d}^2 + C\|g\|_{L^2(\Gamma_1)} + \|f\|_{L^2(\Omega)^d}.
	\end{align*}
	Since RHS is bounded uniformly in $t$, $p$ is in $L^\infty(0, T;\mathring Q)$.
\end{proof}

\subsection{Regularized problem}
To prove the solvability of Problem VI$_\sigma$, we consider a regularized problem VI$_\sigma^\epsilon$, which is shown to be equivalent to a variational equation problem, denoted by Problem VE$_\sigma^\epsilon$.

Before stating those problems in detail, for fixed $\epsilon>0$ we introduce
\begin{equation*}
	j_\epsilon(\eta) = \int_{\Gamma_1}g\rho_\epsilon(\eta)\,ds,
\end{equation*}
where $\rho_\epsilon$ is a regularization of $|\cdot|$ having the following properties:
\begin{enumerate}[\quad(a)]
	\item $\rho_\epsilon \in C^2(\mathbb R^d)$ is a nonnegative convex function.
	\item For all $z\in\mathbb R^d$, it holds that
		\begin{equation}
			\big| \rho_\epsilon(z) - |z| \big| \le \epsilon. \label{c.28}
		\end{equation}
	\item If $\alpha_\epsilon$ denotes $\nabla\rho_\epsilon$, for all $z\in\mathbb R^d$ it holds that 
		\begin{equation}
			|\alpha(z)| \le 1 \quad\text{and}\quad \alpha_\epsilon(z)\cdot z \ge 0. \label{c.18}
		\end{equation}
	\item Let $\beta_\epsilon$ denote the Hessian of $\rho_\epsilon$, namely, $\beta_{\epsilon,ij} = \frac{\partial^2\rho_\epsilon}{\partial z_i\partial z_j}$ for $i,j=1, . . . ,d$.
		Then $\beta_\epsilon$ is semi-positive definite, that is,
		\begin{equation}
			{}^ty \beta_\epsilon(z) y \ge 0 \qquad (\forall y,z\in \mathbb R^d), \label{c.19}
		\end{equation}
		where ${}^t y$ means the transpose of $y$.
		This is a consequence of the convexity of $\rho_\epsilon$.
\end{enumerate}
Such $\rho_\epsilon$ does exist; for example, let $\rho_\epsilon$ be given by $\rho_\epsilon(z) = |z| - (1 - \frac2\pi)\epsilon$ if $|z|\ge\epsilon$, $\rho_\epsilon(z) = \frac{2\epsilon}\pi (1 - \cos\frac\pi{2\epsilon}|z|)$ if $|z|\le\epsilon$.
Then some elementary computation shows that $\rho_\epsilon$ enjoys all of (a)--(d) above.
\begin{Rem}
	One could use the Moreau-Yoshida approximation of $|\cdot|$ as $\rho_\epsilon$, which is considered in \cite{S04}, but it is only in $C^1(\mathbb R^d)$, not in $C^2(\mathbb R^d)$.
\end{Rem}

Since $\rho_\epsilon$ is differentiable, the functional $j_\epsilon$ is G\^ateaux differentiable, with its derivative $Dj_\epsilon(\eta)\in (H^{1/2}(\Gamma_1)^d)'$ computed by
\begin{equation}
	\left< Dj_\epsilon(\eta), \xi \right>_{H^{1/2}(\Gamma_1)^d} = \lim_{h\to0} \frac{j_\epsilon(\eta + h\xi) - j_\epsilon(\eta)}h = \int_{\Gamma_1}g\alpha_\epsilon(\eta)\cdot\xi\,ds \label{c.13}
\end{equation}
for $\eta,\xi \in H^{1/2}(\Gamma_1)^d$.

We are ready to state the regularized problems mentioned above.
\vspace{1mm}
\begin{Prob}{VI$_\sigma^\epsilon$-SBCF}
	For a.e.~$t\in(0,T)$, find $u_\epsilon(t)\in V_{n,\sigma}$  such that $u_\epsilon'(t) \in L^2(\Omega)^d$, $u_\epsilon(0) = u_0^\epsilon$ and 
	\begin{align}
		&(u_\epsilon', v-u_\epsilon) + a_0(u_\epsilon, v-u_\epsilon) + a_1(u_\epsilon, u_\epsilon, v-u_\epsilon) + j_\epsilon(v_\tau) - j_\epsilon(u_{\epsilon\tau}) \notag \\
		\ge &(f, v-u_\epsilon) \hspace{7cm} (\forall v\in V_{n,\sigma}). \label{c.15}
	\end{align}
\end{Prob}
\begin{Prob}{VE$_\sigma^\epsilon$-SBCF}
	For a.e.~$t\in(0,T)$, find $u_\epsilon(t)\in V_{n,\sigma}$  such that $u_\epsilon'(t) \in L^2(\Omega)^d$, $u_\epsilon(0) = u_0^\epsilon$ and 
	\begin{equation}
		(u_\epsilon', v) + a_0(u_\epsilon, v) + a_1(u_\epsilon, u_\epsilon, v) + \int_{\Gamma_1}g\alpha_\epsilon(u_{\epsilon\tau})\cdot v_\tau ds = (f, v) \qquad (\forall v\in V_{n,\sigma}). \label{c.14}
	\end{equation}
\end{Prob}
\noindent Here, $u_0^\epsilon$ is a perturbation of the original initial velocity $u_0$.
The way one obtains $u_0^\epsilon$ from $u_0$ is described later.
By an elementary observation (e.g.\ \cite[Section 3.3]{DuLi72} or \cite[Lemma 3.3]{S04}), we see that:
\begin{Prop}\label{Prop c.1}
	Problems $\mathrm{VI}_\sigma^\epsilon$ and $\mathrm{VE}_\sigma^\epsilon$ are equivalent.
\end{Prop}

Now we focus on the construction of a perturbed initial velocity $u_0^\epsilon$.
Since $u_0\in H^2(\Omega)^d$ satisfies SBCF by (S\ref{S3}), it follows from the Green formula $a_0(u_0, v) = (-\nu\Delta u_0, v) + \int_{\Gamma_1}\sigma_\tau(u_0)\cdot v_\tau\,ds$, for $v\in V_{n,\sigma}$, that
\begin{align}
	&a_0(u_0, v - u_0) + \int_{\Gamma_1} g(0) |v_\tau|\, ds - \int_{\Gamma_1} g(0) |u_{0\tau}|\, ds \ge (-\nu\Delta u_0, v - u_0) \notag \\[-2mm]
	&\hspace{9cm}		(\forall v\in V_{n,\sigma}). \label{c.34}
\end{align}
Here we consider the regularized problem: find $u_0^\epsilon \in V_{n,\sigma}$ such that
\begin{align}
	&a_0(u_0^\epsilon, v - u_0^\epsilon) + \int_{\Gamma_1} g(0)\rho_\epsilon(v_\tau)\, ds - \int_{\Gamma_1} g(0)\rho_\epsilon(u_{0\tau}^\epsilon)\, ds \ge (-\nu\Delta u_0, v - u_0^\epsilon) \notag \\[-2mm]
	&\hspace{9.5cm}	(\forall v\in V_{n,\sigma}), \label{c.25}
\end{align}
which is equivalent to (cf. Proposition \ref{Prop c.1})
\begin{equation}
	a_0(u_0^\epsilon, v) + \int_{\Gamma_1} g(0)\alpha_\epsilon(u_{0\tau}^\epsilon)\cdot v_\tau\, ds = (-\nu\Delta u_0, v) \qquad (\forall v\in V_{n,\sigma}). \label{c.26}
\end{equation}
By a standard theory of elliptic variational inequalities \cite{G84}, (\ref{c.25}) admits a unique solution $u_0^\epsilon$, which is the perturbation of $u_0$ in question.
With this setting, we find:
\begin{Lem} \label{Lem c.1}
	{\rm (i)} When $\epsilon\to0$, $u_0^\epsilon\to u_0$ strongly in $H^1(\Omega)^d$.
	
	{\rm (ii)} $u_0^\epsilon \in H^2(\Omega)^d$ and
	\begin{equation}
		\|u_0^\epsilon\|_{H^2(\Omega)^d} \le C(\|\nu\Delta u_0\| + \|g(0)\|_{H^1(\Gamma_1)}). \label{c.27}
	\end{equation}
\end{Lem}
\begin{proof}
	(i) Taking $v = u_0$ in (\ref{c.25}) and $v = u_0^\epsilon$ in (\ref{c.34}), adding the resulting two inequalities, applying Korn's inequality, and using (\ref{c.28}), we conclude
	\begin{align*}
		\alpha \|u_0^\epsilon - u_0\|_{H^1(\Omega)^d}^2 &\le \int_{\Gamma_1} g(0)\big( |u_0^\epsilon| - \rho_\epsilon(u_0^\epsilon) \big)ds + \int_{\Gamma_1} g(0)\big( \rho_\epsilon(u_0) - |u_0| \big)ds \\
			&\le 2\epsilon \int_{\Gamma_1} g(0)\,ds \to 0 \qquad (\epsilon\to0).
	\end{align*}
	
	(ii) Since $g(0) \in H^1(\Gamma_1)$ by (S\ref{S2}), we can directly apply the regularity result \cite[Lemma 5.2]{S04} to the elliptic variational inequality (\ref{c.25}), and obtain (\ref{c.27}).
	Though our $\rho_\epsilon$ and $\alpha_\epsilon$ are different from those of \cite{S04}, it makes no difference in the proof of that lemma.
\end{proof}
\begin{Rem} \label{Rem c.10}
	(i) As a result of (i) above, for sufficiently small $\epsilon>0$ we have
	\begin{equation}
		\|u_0^\epsilon\|_{L^2(\Omega)^d} \le 2\|u_0\|_{L^2(\Omega)^d} \qquad\text{and}\qquad \|u_0^\epsilon\|_{H^1(\Omega)^d} \le 2\|u_0\|_{H^1(\Omega)^d}. \label{c.51}
	\end{equation}
	
	(ii) Concerning the regularity of the domain, \cite{S04} assumes that $\Gamma_0$ and $\Gamma_1$ are class of $C^2$ and $C^4$ respectively, which is sufficient for our theory as well.
\end{Rem}
\begin{Rem}
	In \cite{S04}, dealing with the stationary problem, the author stated that $g\in H^{1/2}(\Gamma_1)$ was enough to derive $u\in H^2(\Omega)^d$ and $p\in H^1(\Omega)$.
	However, it turned out that his proof presented there worked only for $g\in H^1(\Gamma_1)$; see the errata by the same author.
	This is why we have assumed $g(0)\in H^1(\Gamma_1)$ in (S\ref{S2}), not $g(0)\in H^{1/2}(\Gamma_1)$.
\end{Rem}

\subsection{Proof of existence} \label{Sec 3.4}
Due to Proposition \ref{Prop c.1}, we concentrate on solving Problem VE$_\sigma^\epsilon$.
In doing so, we construct approximate solutions by Galerkin's method.
Since $V_{n,\sigma}\subset H^1(\Omega)^d$ is separable, there exist members $w_1,w_2,. . .\, \in V_{n,\sigma}$, linear independent to each other, such that $\bigcup_{m=1}^\infty \mathrm{span}\{w_k\}_{k=1}^m \subset V_{n,\sigma}$ dense in $H^1(\Omega)^d$.
Here $\epsilon$ is fixed, and thus we may assume $w_1 = u_0^\epsilon$.

\vspace{2mm}
\begin{Prob}{VE$_\sigma^{\epsilon,m}$-SBCF}
	Find $c_k\in C^2([0, T])\, (k=1, . . . ,m)$ such that $u_m\in V_{n,\sigma}$ defined by $u_m = \sum_{k=1}^m c_k(t)w_k$ satisfies $u_m(0) = u_0^\epsilon$ and
	\begin{align}
		&(u_m', w_k) + a_0(u_m, w_k) + a_1(u_m, u_m, w_k) + \int_{\Gamma_1} g\alpha_\epsilon(u_{m\tau})\cdot w_{k\tau}ds = (f, w_k) \notag \\[-2mm]
		&\hspace{9cm}		(k=1, . . . , m). \label{c.17}
	\end{align}
\end{Prob}
Since $\alpha_\epsilon\in C^1(\mathbb R^d)^d$, the system of ordinal differential equations (\ref{c.17}) admits unique solutions $c_k\in C^2([0, \tilde T])\, (k=1, . . . ,m)$ for some $\tilde T\le T$.
The a priori estimate below shows $\tilde T$ can be taken as $T$, so that we write $T$ instead of $\tilde T$ from the beginning.
\begin{Prop} \label{Prop c.2}
	Let {\rm (S\ref{S1})--(S\ref{S3})} be valid and $\epsilon$ be small enough so that $(\ref{c.51})$ holds.
	
	{\rm (i)} When $d=2$, $u_m\in L^\infty(0, T; V_{n,\sigma})$ and $u_m' \in L^\infty(0, T; L^2(\Omega)^d)\cap L^2(0, T; V_{n,\sigma})$ are bounded independently of $m$ and $\epsilon$.
	
	{\rm (ii)} When $d=3$, the same conclusion holds for some smaller interval $(0, T')$, which can be taken independently of $m$ and $\epsilon$.
\end{Prop}
\begin{proof}
	Due to space limitations, we simply write $\|u\|_{L^2}$, $\|g\|_{L^2}$, $\|f\|_{L^2}$, . . . instead of $\|u\|_{L^2(\Omega)^d}$, $\|g\|_{L^2(\Gamma_1)}$, $\|f\|_{L^2(\Omega)^d}$, . . . and so on.
	
	(i) Multiplying (\ref{c.17}) by $c_k(t)$, and adding the resulting equations for $k=1, . . . , m$, we obtain
	\begin{equation*}
		(u_m', u_m) + a_0(u_m, u_m) + \int_{\Gamma_1} g\alpha_\epsilon(u_{m\tau})\cdot u_{m\tau} ds = (f, u_m),
	\end{equation*}
	where we have used Lemma \ref{Lem b.2}(i). It follows from (\ref{b.3}) and (\ref{c.18}) that
	\begin{equation*}
		\frac12\frac{d}{dt} \|u_m\|_{L^2}^2 + \alpha \|u_m\|_{H^1}^2 \le (f, u_m) \le \|f\|_{L^2} \|u_m\|_{H^1} \le \frac\alpha2 \|u_m\|_{H^1}^2 + \frac1{2\alpha} \|f\|_{L^2}^2,
	\end{equation*}
	which gives
	\begin{equation}
		\frac{d}{dt} \|u_m\|_{L^2}^2 + \alpha \|u_m\|_{H^1}^2 \le C\|f\|_{L^2}^2. \label{c.23}
	\end{equation}
	Consequently, for $0\le t \le T$,
	\begin{equation} 
		\|u_m(t)\|_{L^2}^2 + \alpha \int_0^T \|u_m\|_{H^1}^2 dt \le \|u_0^\epsilon\|_{L^2}^2 + C\int_0^T \|f\|_{L^2}^2 dt. \label{c.35}
	\end{equation}
	Since $\|u_0^\epsilon\|_{L^2(\Omega)^d} \le 2\|u_0\|_{L^2(\Omega)^d}$ by assumption, we find that $\|u_m\|_{L^\infty(0, T; L^2)}$ and $\|u_m\|_{L^2(0, T; V_{n,\sigma})}$ are bounded by $C(f,u_0)$ independently of $m$ and $\epsilon$.
	
	Next, we differentiate $(\ref{c.17})$ with respect to $t$, which is possible because $c_k(t)$'s are in $C^2([0, T])$, to deduce
	\begin{align*}
		&(u_m'', w_k) + a_0(u_m', w_k) + a_1(u_m', u_m, w_k) + a_1(u_m, u_m', w_k) \\
		&\hspace{8mm}	+ \int_{\Gamma_1} g' \alpha_\epsilon(u_{m\tau})\cdot w_{k\tau} ds + \int_{\Gamma_1} g\;^t u_{m\tau}' \beta_\epsilon w_{k\tau}\, ds = (f', w_k) \quad (k=1, . . . , m).
	\end{align*}
	Multiplying this by $c_k'(t)$, and adding the resulting equations, we obtain
	\begin{align}
		&(u_m'', u_m') + a_0(u_m', u_m') + a_1(u_m', u_m, u_m') + \int_{\Gamma_1} g' \alpha_\epsilon(u_{m\tau})\cdot u_{m\tau}' ds \notag \\
		&\hspace{4cm}	+ \int_{\Gamma_1} g\;^t u_{m\tau}' \beta_\epsilon(u_{m\tau}) u_{m\tau}'\, ds = (f', u_m'), \label{c.20}
	\end{align}
	where we have again used Lemma \ref{Lem b.2}(i). Here,
	{\allowdisplaybreaks
	\begin{align}
		&&\hspace{-3cm}a_1(u_m', u_m, u_m') &\le C\|u_m'\|_{L^2} \|u_m\|_{H^1} \|u_m'\|_{H^1} \qquad(\text{by } (\ref{b.5})) \notag \\
			&&\hspace{-3cm}&\le \frac\alpha6 \|u_m'\|_{H^1}^2 + C \|u_m\|_{H^1}^2 \|u_m'\|_{L^2}^2, \label{c.29} \\[1mm]
		&&\hspace{-3cm}\left| \int_{\Gamma_1} g' \alpha_\epsilon(u_{m\tau})\cdot u_{m\tau}' ds \right| &\le \|g'\|_{L^2} \|u_{m\tau}'\|_{L^2(\Gamma_1)^d} \qquad(\text{by } (\ref{c.18})) \notag \\
			&&\hspace{-3cm}&\le C\|g'\|_{L^2} \|u_m'\|_{H^1} \qquad(\text{by Lemma \ref{Lem b.1}(i)}) \notag \\
			&&\hspace{-3cm}&\le \frac\alpha6 \|u_m'\|_{H^1}^2 + C\|g'\|_{L^2}^2, \notag \\[1mm]
		&&\hspace{-3cm}\int_{\Gamma_1} g\;^t u_{m\tau}' \beta_\epsilon(u_{m\tau}) u_{m\tau}'\, ds &\ge 0, \qquad(\text{by $g>0$ and (\ref{c.19})}) \notag \\[1mm]
		&&\hspace{-3cm}|(f', u_m')| &\le \|f'\|_{L^2} \|u_m'\|_{H^1} \le \frac\alpha6 \|u_m'\|_{H^1}^2 + C\|f'\|_{L^2}^2. \notag
	\end{align}
	}
	Collecting these estimates, it follows from (\ref{c.20}) that for $0\le t\le T$
	\begin{equation}
		\frac{d}{dt} \|u_m'\|_{L^2}^2 + \alpha \|u_m'\|_{H^1}^2 \le C( \|f'\|_{L^2}^2 + \|g'\|_{L^2}^2 ) + C \|u_m\|_{H^1}^2 \|u_m'\|_{L^2}^2. \label{c.22}
	\end{equation}
	If the second term of LHS is neglected, Gronwall's inequality leads to
	\begin{equation}
		\|u_m'(t)\|_{L^2}^2 \le \left( \|u_m'(0)\|_{L^2}^2 + C\int_0^T ( \|f'\|_{L^2}^2 + \|g'\|_{L^2}^2 ) dt \right) e^{C\int_0^T \|u_m\|_{H^1}^2 dt}. \label{c.21}
	\end{equation}
	Provided that $\|u_m'(0)\|_{L^2}^2$ is bounded independently of $m$ and $\epsilon$, estimate (\ref{c.21}) gives the boundedness of $\|u_m'\|_{L^\infty(0, T; L^2)}$ because we know that of $\|u_m\|_{L^2(0, T; V_{n,\sigma})}$ due to (\ref{c.35}).
	Then, by (\ref{c.23}) and (\ref{c.35}) we have
	\begin{equation*}
		\alpha\|u_m(t)\|_{H^1}^2 \le C\|f\|_{L^2}^2 - \|u_m'\|_{L^2}\|u_m\|_{L^2} \le C(f, g, u_0),
	\end{equation*}
	which implies $\|u_m\|_{L^\infty(0, T; V_{n,\sigma})}$ is bounded.
	Finally, integrating (\ref{c.22}), we see that $\|u_m'\|_{L^2(0, T; V_{n,\sigma})}$ is also bounded.
	
	To show the boundedness of $\|u_m'(0)\|_{L^2}^2$, we multiply (\ref{c.17}) by $c_k'(t)$, add the resulting equations, and make $t=0$, arriving at
	\begin{align}
		&\|u_m'(0)\|_{L^2}^2 + a_0( u_0^\epsilon, u_m'(0) ) + a_1( u_0^\epsilon, u_0^\epsilon, u_m'(0) ) + \int_{\Gamma_1} g(0)\alpha_\epsilon(u_{0\tau}^\epsilon)\cdot u_{m\tau}'(0)ds \notag \\
		&	= (f(0), u_m'(0)). \label{c.24}
	\end{align}
	From the construction of $u_0^\epsilon$, especially (\ref{c.26}), we have
	\begin{align}
		\left| a_0(u_0^\epsilon, u_m'(0)) + \int_{\Gamma_1} g(0)\alpha_\epsilon(u_{0\tau}^\epsilon)\cdot u_{m\tau}'(0)ds \right|  &=  |(-\nu\Delta u_0, u_m'(0))| \notag \\
			&\le  C\|u_0\|_{H^2} \|u_m'(0)\|_{L^2}. \label{c.60}
	\end{align}
	Furthermore, by Schwarz's inequality, Sobolev's inequality and (\ref{c.27}),
	\begin{align*}
		 |a_1(u_0^\epsilon, u_0^\epsilon, u_m'(0))| &\le C\|u_0^\epsilon\|_{L^\infty} \|u_0^\epsilon\|_{H^1} \|u_m'(0)\|_{L^2} \le C \|u_0^\epsilon\|_{H^2}^2\|u_m'(0)\|_{L^2} \\
		 	&\le C (\|u_0\|_{H^2} + \|g(0)\|_{H^1} )^2 \|u_m'(0)\|_{L^2}.
	\end{align*}
	Combining these estimates with (\ref{c.24}), we obtain
	\begin{equation*}
		\|u_m'(0)\|_{L^2} \le \|f(0)\|_{L^2} + C\|u_0\|_{H^2} + C( \|u_0\|_{H^2} + \|g(0)\|_{H^1} )^2,
	\end{equation*}
	which proves the boundedness of $\|u_m'(0)\|_{L^2}^2$.
	This completes the proof of (i).

	(ii) The discussion before (\ref{c.29}) and the observation for $\|u_m'(0)\|_{L^2}$ are the same as (i).
	What changes from the case $d=2$ is that when $d=3$, instead of (\ref{c.29}), we only have (by (\ref{b.6}) and Young's inequality)
	\begin{align*}
		|a_1(u_m', u_m, u_m')| &\le C \|u_m'\|_{L^2}^{1/2} \|u_m\|_{H^1} \|u_m'\|_{H^1}^{3/2} \\
			&\le \gamma\|u_m\|_{H^1} \|u_m'\|_{H^1}^2  \!+\!  C\|u_m\|_{H^1} \|u_m'\|_{L^2}^2,
	\end{align*}
	for a constant $\gamma>0$ which can be arbitrarily small.
	We choose $\gamma$ satisfying $\gamma\|u_0\|_{H^1} \le \frac\alpha{24}$, and from (\ref{c.51}) we obtain $\gamma\|u_0^\epsilon\|_{H^1} \le \frac\alpha{12}$. 
	Let $T'>0$, which may depend on $m,\epsilon$ at this stage, be the maximum value of $t$ such that $\gamma\|u_m(t)\|_{H^1} \le \frac\alpha6$.
	If $\gamma\|u_m(t)\|_{H^1} < \frac\alpha6$ for all $0\le t\le T$, we set $T'=T$.
	Since $\gamma\|u_m(0)\|_{H^1} < \frac\alpha6$ and $u_m(t)$ is continuous with respect to $t$, such $T'$ does exist, and furthermore if $T'<T$ then $\gamma\|u_m(t)\|_{H^1} = \frac\alpha6$.
	
	Therefore, in place of (\ref{c.22}) we obtain
	\begin{equation*}
		\frac{d}{dt} \|u_m'\|_{L^2}^2 + \alpha \|u_m'\|_{H^1}^2 \le C( \|f'\|_{L^2}^2 + \|g'\|_{L^2}^2 ) + C \|u_m\|_{H^1} \|u_m'\|_{L^2}^2 \quad (0\le t\le T'),
	\end{equation*}
	which leads to the boundedness of $\|u_m'\|_{L^2(0, T'; V_{n,\sigma})}$ and $\|u_m'\|_{L^\infty(0, T'; L^2)}$, together with $\|u_m\|_{L^\infty(0, T'; V_{n,\sigma})}$.
	
	Finally, let us prove that $T'$ is bounded from below independently of $m$ and $\epsilon$. In fact, if $T'<T$ then we see that
	\begin{align*}
		\frac\alpha{12\gamma} &\le \|u_m(T')\|_{H^1} - \|u_m(0)\|_{H^1} \le \|u_m(T') - u_m(0)\|_{H^1} = \left\| \int_0^{T'} u_m'(t)\,dt \right\|_{H^1} \\
			&\le \int_0^{T'} \|u_m'(t)\|_{H^1} dt \le \sqrt{T'} \|u_m'\|_{L^2(0, T'; V_{n,\sigma})}.
	\end{align*}
	Since we already know $\|u_m'\|_{L^2(0, T'; V_{n,\sigma})}$ is bounded, we obtain the lower bound for $T'$.
	This completes the proof of Proposition \ref{Prop c.2}.
\end{proof}
\begin{Rem}
	(i) A naive computation gives, by (\ref{c.18}),
	\begin{equation*}
		\left| \int_{\Gamma_1} g(0)\alpha_\epsilon(u_{0\tau}^\epsilon)\cdot u_{m\tau}'(0)\, ds \right| \le \|g(0)\|_{L^2(\Gamma_1)} \|u_{m\tau}'(0)\|_{L^2(\Gamma_1)^d},
	\end{equation*}
	but $\|u_{m\tau}'(0)\|_{L^2(\Gamma_1)^d}$ cannot be bounded by $\|u_m'(0)\|_{L^2(\Omega)^d}$ in general.
	Therefore, the perturbation of $u_0$, which is based on the compatibility condition in (S\ref{S3}), is essential in deriving (\ref{c.60}).
	
	(ii) If $d=3$ and $f$, $g$, $u_0$ are sufficiently small, we can prove $\gamma\|u_m(t)\|_{H^1(\Omega)^d}\\ \le \frac\alpha6$ for all $0\le t\le T$, and consequently the existence of a global solution.
\end{Rem}

As a final step for our proof of the existence, we discuss passing to the limits $m\to\infty$ and $\epsilon\to0$.
The proof below is valid for both $d=2,3$, except that when $d=3$ we have to replace $T$ with $T'$ given in Proposition \ref{Prop c.2}.
\begin{Prop}\label{Prop c.3}
	{\rm (i)} Under the assumptions of Proposition $\ref{Prop c.2}$, there exists a solution $u_\epsilon$ of Problem $\mathrm{VI}_\sigma^\epsilon$ such that all of
	$\|u_\epsilon\|_{L^\infty(0, T; V_{n,\sigma})}$, $\|u_\epsilon'\|_{L^2(0, T; V_{n,\sigma})}$, and $\|u_\epsilon'\|_{L^\infty(0, T; L^2(\Omega)^d)}$ are bounded independently of $\epsilon$.
	
	{\rm (ii)} There exists a strong solution of Problem $\mathrm{VI}_\sigma$.
\end{Prop}
\begin{proof}
	(i) As a consequence of Proposition \ref{Prop c.2}, there exists some $u_\epsilon$ and a subsequence of $\{u_m\}_{m=1}^\infty$, denoted again by $\{u_m\}_{m=1}^\infty$, such that
	$u_\epsilon \in L^\infty(0, T; V_{n,\sigma})$, $u_\epsilon' \in L^2(0, T; V_{n,\sigma}) \cap L^\infty(0, T; L^2(\Omega)^d)$, and when $m\to\infty$
	\begin{align*}
		&u_m \rightharpoonup u_\epsilon &&\text{weakly-$*$ in } L^\infty(0, T; V_{n,\sigma}), \\
		&u_m' \rightharpoonup u_\epsilon' &&\text{weakly in } L^2(0,T; V_{n,\sigma}) \text{ and weakly-$*$ in } L^\infty(0, T; L^2(\Omega)^d).
	\end{align*}
	We also find that all of $\|u_\epsilon\|_{L^\infty(0, T; V_{n,\sigma})}$, $\|u_\epsilon'\|_{L^2(0, T; V_{n,\sigma})}$, and $\|u_\epsilon'\|_{L^\infty(0, T; L^2(\Omega)^d)}$ are bounded independently of $\epsilon$.

	Let us prove $u_\epsilon$ solves Problem VI$_\sigma^\epsilon$. By Proposition \ref{Prop c.1}, it suffices to show that $u_\epsilon$ is a solution of Problem VE$_\sigma^\epsilon$.
	Multiplying (\ref{c.17}) by an arbitrary $\phi\in C_0^\infty(0,T)$ and integrating over $(0, T)$, we obtain 
	\begin{align}
		&\int_0^T \phi(t) \Big\{ (u_m', w_k) + a_0(u_m, w_k) + a_1(u_m, u_m, w_k) + \int_{\Gamma_1} g\alpha_\epsilon(u_{m\tau})\cdot w_{k\tau}ds \notag \\
		&\hspace{5cm}	- (f, w_k) \Big\}dt = 0 \qquad (k = 1, . . . , m). \label{c.30}
	\end{align}
	It follows from \cite[Theorem III.2.1]{Tem77} that the embedding
	\begin{equation*}
		\{ v \,\big|\, v\in L^2(0,T; H^1(\Omega)^d),\, v' \in L^2(\Omega\times (0,T))^d \} \hookrightarrow L^2(0,T; L^4(\Omega)^d)
	\end{equation*}
	is compact, so that $u_m\to u_\epsilon$ strongly in $L^2(0,T; L^4(\Omega)^d)$.
	Moreover, since the trace operator $H^1(\Omega\times (0,T))\to L^2(\Gamma_1\times(0, T))$ is compact, $u_{m\tau}\to u_{\epsilon\tau}$ strongly in $L^2(\Gamma_1\times(0, T))^d$.
	In particular, $u_{m\tau}\to u_{\epsilon\tau}$ a.e.~on $\Gamma_1\times (0, T)$, and thus the continuity of $\alpha_\epsilon(z)$ yields $\alpha_\epsilon(u_{m\tau}) \to \alpha_\epsilon(u_{\epsilon\tau})$  a.e.~on $\Gamma_1\times (0, T)$.
	Making $m\to\infty$ in (\ref{c.30}), together with Lebesgue's convergence theorem, we see that
	\begin{align*}
		&\int_0^T \phi(t) \Big\{ (u_\epsilon', w_k) + a_0(u_\epsilon, w_k) + a_1(u_\epsilon, u_\epsilon, w_k) + \int_{\Gamma_1} g\alpha_\epsilon(u_{\epsilon\tau})\cdot w_{k\tau} ds \\
		&\hspace{7cm}	- (f, w_k) \Big\}dt = 0 \qquad (k=1,2,. . .).
	\end{align*}
	Since $\overline{ \bigcup_{m=1}^\infty \mathrm{span}\{w_k\}_{k=1}^m } = V_{n,\sigma}$, the above equation is valid for all test functions $v\in V_{n,\sigma}$.
	Hence (\ref{c.14}) holds for a.e. $t\in (0, T)$, which implies that $u_\epsilon$ is a solution of Problem VE$_\sigma^\epsilon$.
	
	(ii) As a result of (i), there exists some $u$ and a sequence $\epsilon_l\to0$ $(l\to\infty)$, to which we drop the subscript $l$ for simplicity, such that
	$u \in L^\infty(0, T; V_{n,\sigma})$, $u' \in L^2(0, T; V_{n,\sigma}) \cap L^\infty(0, T; L^2(\Omega)^d)$, and when $\epsilon\to0$
	\begin{align*}
		&u_\epsilon \rightharpoonup u &&\text{weakly-$*$ in } L^\infty(0, T; V_{n,\sigma}), \\
		&u_\epsilon' \rightharpoonup u' &&\text{weakly in } L^2(0,T; V_{n,\sigma}) \text{ and weakly-$*$ in } L^\infty(0, T; L^2(\Omega)^d).
	\end{align*}
	As before, one sees that $u_\epsilon\to u$ strongly in $L^2(0,T; L^4(\Omega)^d)$ and $u_{\epsilon\tau}\to u_\tau$ strongly in $L^2(\Gamma_1\times(0, T))$.
	In addition, $u_\epsilon \rightharpoonup u$ weakly in $L^2(0, T; V_{n,\sigma})$, and thus it follows that $\int_0^T a_0(u, u)\,dt \le  \varliminf_{\epsilon\to0} \int_0^T a_0(u_\epsilon, u_\epsilon)\,dt$.
	
	Following the technique of \cite[p.56]{DuLi72}, we let $\tilde v(t)\in L^2(0, T; V_{n,\sigma})$ be arbitrary.
	For a.e.~$t\in (0, T)$, we take $v=\tilde v(t)$ in (\ref{c.15}) and integrate the resulting equation over $(0, T)$ to deduce
	\begin{align}
		&\int_0^T \Big\{ (u_\epsilon', \tilde v - u_\epsilon) + a_0(u_\epsilon, \tilde v - u_\epsilon) + a_1(u_\epsilon, u_\epsilon, \tilde v - u_\epsilon) \notag \\[-2mm]
		&\hspace{5cm}	+ j_\epsilon(\tilde v_\tau) - j_\epsilon(u_{\epsilon\tau}) - (f, \tilde v - u_\epsilon) \Big\}dt \ge 0. \label{c.31}
	\end{align}
	In view of (\ref{c.28}), together with triangle inequality and Lipschitz continuity of $j$, we have $\int_0^T j_\epsilon(\tilde v_\tau)\,dt \to \int_0^T j(\tilde v_\tau)\,dt$ and $\int_0^T j_\epsilon(u_{\epsilon\tau})\,dt \to \int_0^T j(u_\tau)\,dt$ when $\epsilon\to0$.
	Therefore, taking the lower limit $\varliminf_{\epsilon\to0}$ in (\ref{c.31}) gives
	\begin{equation*}
		\int_0^T \Big\{ (u', \tilde v - u) + a_0(u, \tilde v - u) + a_1(u, u, \tilde v - u) + j(\tilde v_\tau) - j(u_\tau) - (f, \tilde v - u) \Big\}dt \ge 0.
	\end{equation*}
	A technique using the Lebesgue differentiation theorem allows us to conclude that $u$ satisfies (\ref{c.4}) at a.e.~$t=t_0$ (for more detail, see \cite[p.57]{DuLi72}).
	
	Concerning the initial condition, since the trace operator $H^1(\Omega\times (0, T)) \to L^2(\Omega\times \{0\})$ is continuous, Lemma \ref{Lem c.1}(i) leads to $u(0) = \lim_{\epsilon\to0}u_\epsilon(0) = \lim_{\epsilon\to0}u_0^\epsilon = u_0$.
	Hence $u$ is a strong solution of Problem VI$_\sigma$.
\end{proof}

Propositions \ref{Prop c.4} and \ref{Prop c.3}(ii) complete the proof of Theorem \ref{Thm c.1}.

\section{Navier-Stokes Problem with LBCF}
\subsection{Weak formulations}
Throughout this section, we assume $f\in L^2(\Omega\times (0,T))$, $u_0\in V_{\tau,\sigma}$, and $g\in L^2(\Gamma_1\times(0,T))$.
Further regularity assumptions on these data will be given before Theorem \ref{Thm d.1}.
In addition, the barrier term $j_n$ is simply written as $j$.
A primal weak formulation of (\ref{a.3})--(\ref{a.5}) with (\ref{a.2}) is as follows:
\vspace{2mm}
\begin{Prob}{PDE-LBCF}
	For a.e.~$t\in(0,T)$, find $(u(t), p(t))\in V_\tau\times Q$ such that $u'(t) \in L^2(\Omega)^d$, $u(0) = u_0$, $\sigma_n$ is well-defined in the sense of Definition \ref{Def b.1}, $|\sigma_n| \le g$ a.e.~on $\Gamma_1$, and $\sigma_n u_n + g|u_n| = 0$ a.e.~on $\Gamma_1$.
\end{Prob}
\begin{Rem}
	More precisely, ``$|\sigma_n| \le g$" implies that $\sigma_n \in (H^{1/2}(\Gamma_1))'$ actually belongs to $L^\infty_{1/g}(\Gamma_1)$ with $\|\sigma_n\|_{L^\infty_{1/g}(\Gamma_1)} \le 1$.
	In particular, $\sigma_n \in L^2(\Gamma_1)$.
\end{Rem}
Throughout this section, we refer to Problem PDE-LBCF just as Problem PDE.
Similar abbreviation will be made for other problems.
Next, as in SBCF, we propose:
\begin{Prob}{VI-LBCF}
	For a.e.~$t\in(0,T)$, find $(u(t), p(t))\in V_\tau\times Q$, such that $u'(t) \in L^2(\Omega)^d$, $u(0) = u_0$ and
	\begin{numcases}{\hspace{-1cm}}
		(u', v - u) + a_0(u, v - u) + a_1(u, u, v - u) + b(v - u, p) \notag \\
		\hspace{5cm}	+ j(v_n) - j(u_n) \ge (f, v - u) &$(\forall v\in V_\tau)$, \label{d.3} \\
		b(u,q) = 0 &$(\forall q\in Q)$.
	\end{numcases}
\end{Prob}
\begin{Prob}{VI$_\sigma$-LBCF}
	For a.e.~$t\in(0,T)$, find $u(t)\in V_{\tau,\sigma}$  such that $u'(t) \in L^2(\Omega)^d$, $u(0) = u_0$ and 
	\begin{equation}
		(u', v-u) + a_0(u, v-u) + a_1(u, u, v-u) + j(v_n) - j(u_n) \ge (f, v-u) \quad (\forall v\in V_{\tau,\sigma}). \label{d.5}
	\end{equation}
\end{Prob}
Unlike the case of SBCF, Problem VI$_\sigma$ is not exactly equivalent to Problem PDE, as is shown in the following theorem.
\begin{Thm}
	{\rm (i)} If $(u,p)$ solves Problem $\mathrm{PDE}$, then $u$ solves Problem $\mathrm{VI}_\sigma$.
	
	{\rm (ii)} If $u$ solves Problem $\mathrm{VI}_\sigma$, then there exists at least one $p$ such that $(u,p)$ solves Problem $\mathrm{PDE}$.
	If another $p^*$ satisfies the same condition, then for a.e.~$t\in(0, T)$ there exists a unique $\delta(t)\in\mathbb R$ such that
	\begin{equation}
		p(t) = p^*(t) + \delta(t) \quad\text{and}\quad \sigma_n(u(t), p(t)) = \sigma_n(u(t), p^*(t)) - \delta(t). \label{d.9}
	\end{equation}
	
	{\rm (iii)} In {\rm (ii)}, if we assume furthermore $u_n(t)\neq0$, then $\delta(t) = 0$. Namely, the associated pressure is uniquely determined.
\end{Thm}
\begin{proof}
	(i) This can be proved by the same way as Theorem \ref{Thm c.2}.
	
	(ii) For a.e.~$t\in (0, T)$ and $v\in\mathring V_\sigma$, it follows from (\ref{d.5}) that $(u', v) + a_0(u, v) + a_1(u, u, v) = (f, v)$, and thus there exists unique $\mathring p \in\mathring Q$ such that
	\begin{equation*}
		(u', v) + a_0(u, v) + a_1(u, u, v) + b(v,\mathring p) = (f, v) \qquad (\forall v\in\mathring V).
	\end{equation*}
	According to Definition \ref{Def b.1}, $\mathring\sigma_n = \sigma_n(u,\mathring p)$ is well-defined, so that
	\begin{equation*}
		(u',v) + a_0(u,v) + b(v, \mathring p) + a_1(u,u,v) - \left<\mathring\sigma_n, v_n\right>_{H^{1/2}(\Gamma_1)} = (f,v) \qquad (\forall v\in V_\tau).
	\end{equation*}
	Substituting this equation into (\ref{d.5}), we obtain $- \langle \mathring\sigma_n, v_n - u_n \rangle_{H^{1/2}(\Gamma_1)} \le j(v_n) - j(u_n)$ for all $v \in V_{\tau,\sigma}$.
	It follows from Lemma \ref{Lem b.10}(ii) that
	\begin{equation*}
		|\langle \mathring\sigma_n, \eta \rangle_{H^{1/2}(\Gamma_1)}| \le \int_{\Gamma_1}g|\eta|\,ds \qquad (\forall\eta\in H^{1/2}(\Gamma_1) \cap L^2_0(\Gamma_1)).
	\end{equation*}
	The Hahn-Banach theorem allows us to extend $\mathring\sigma_n$ to a linear functional $\sigma_n: L^1_g(\Gamma_1)\to\mathbb R$ satisfying the same inequality as above for all $\eta\in L^1_g(\Gamma_1)$.
	Therefore, $\sigma_n \in L^\infty_{1/g}(\Gamma_1)$ and $|\sigma_n| \le g$. In addition, $\sigma_nu_n + g|u_n| = 0$ follows.
	
	Since $\mathring\sigma_n - \sigma_n$ vanishes on $H^{1/2}(\Gamma_1) \cap L^2_0(\Gamma_1)$, there exists a constant $\delta(t)$ such that $\mathring\sigma_n - \sigma_n = \delta(t)$.
	Now, by setting $p(t) =\mathring p(t) + \delta(t)$, it follows that $\sigma_n$ given above actually equals $\sigma_n(u(t), p(t))$ and that $(u(t), p(t))$ solves Problem PDE.
	Relation (\ref{d.9}) can be verified by a similar argument.
	
	(iii) Since $\int_{\Gamma_1}u_n\,ds = \int_\Omega \mathrm{div}\,u\,dx = 0$, the assumption $u_n(t)\neq0$ implies that there exist subsets $A_+, A_-$ of $\Gamma_1$ with positive $d-1$ dimensional Lebesgue measure satisfying $u_n(t) > 0$ on $A_+$ and $u_n(t) < 0$ on $A_-$.
	Because $|\sigma_n| \le g$ and $\sigma_nu_n + g|u_n| = 0$ on $\Gamma_1$, $\sigma_n = -g(t)$ on $A_+$ and $\sigma_n = g(t)$ on $A_-$.
	Hence $\delta(t)$ in (\ref{d.9}) cannot be other than zero.
\end{proof}
\begin{Rem}
	Since $|\sigma_n| \le g$, $\delta(t)$ is no more than $2g(t)$ nor less than $-2g(t)$.
\end{Rem}

\subsection{Main theorem}
Let us state our main theorems for the case of LBCF.
As in SBCF, some compatibility condition is necessary; it is rather complicated because normal stress at $t=0$ involves  a pressure at $t=0$, which is not given as a data.
The precise description is as follows: we say that LBCF is satisfied at $t=0$ if $u_0\in H^2(\Omega)\cap V_{\tau,\sigma}$ and there exists $p_0\in H^1(\Omega)^d$ such that
\begin{equation}
	|\sigma_n(u_0, p_0)| \le g(0)  \quad\text{and}\quad \sigma_n(u_0, p_0)u_{0n} + g(0)|u_{0n}| = 0 \quad\text{a.e. on }\; \Gamma_1 \label{d.53}
\end{equation}
We remark that a similar compatibility condition appears in nonlinear semigroup approaches (see \cite{Fuj01,Fuj02}).

Furthermore, in order to overcome a difficulty arising from (\ref{a.4}), we need no-leak condition at $t=0$, that is, $u_{0n}=0$ on $\Gamma_1$.
In view of (\ref{d.53}), this is automatically satisfied if $|\sigma_n(u_0, p_0)| < g(0)$ on $\Gamma_1$.
Examining our proof of the a priori estimates carefully, one finds that this assumption can be weaken to the condition that $\|u_{0n}\|_{L^2(\Gamma_1)}$ is sufficiently small.

Including what we have discussed above, we assume the followings:
\begin{enumerate}[\qquad(L1)]
	\item $f\in H^1(0,T; L^2(\Omega)^d)$. \label{L1}
	\item $g\in H^1(0, T; L^2(\Gamma_1))$ with $g(0)\in H^1(\Gamma_1)$. \label{L2}
	\item $u_0\in H^2(\Omega)^d\cap V_{\tau,\sigma}$, and LBCF is satisfied at $t=0$. \label{L3}
	\item $u_{0n} = 0$ a.e.~on $\Gamma_1$. \label{L4}
\end{enumerate}

\begin{Thm}\label{Thm d.1}
	Under {\rm (L\ref{L1})--(L\ref{L4})} above, there exists a unique solution $u$ of Problem $\mathrm{VI}_\sigma$ on some interval $(0, T')$, with $T'\le T$, such that
	\begin{equation*}
		u\in L^\infty(0,T'; V_{\tau,\sigma}), \qquad u'\in L^\infty(0,T'; L^2(\Omega)^d)\cap L^2(0,T'; V_{\tau,\sigma}).
	\end{equation*}
\end{Thm}

The uniqueness can be proved by the same way as Proposition \ref{Prop c.4}.
We can also obtain $p\in L^\infty(0, T'; L^2(\Omega))$ by a similar manner to Proposition \ref{Prop c.5}, using the rather infamous inf-sup condition (see \cite[Lemma 2.2]{S04})
\begin{equation*}
	C \|p\|_{L^2(\Omega)} \le \sup_{v\in V_\tau} \frac{b(v, p)}{\|v\|_{H^1(\Omega)^d}} \qquad (\forall p\in L^2(\Omega)).
\end{equation*}

The rest of this section is devoted to the proof of the existence. To state regularized problems, for fixed $\epsilon>0$ we introduce
\begin{equation*}
	j_\epsilon(\eta) = \int_{\Gamma_1}g\rho_\epsilon(\eta)\,ds,
\end{equation*}
where $\rho_\epsilon$ is a function such that
\begin{enumerate}[\quad(a)]
	\item $\rho_\epsilon \in C^2(\mathbb R)$ is a nonnegative convex function.
	\item For all $z\in\mathbb R$, $\big| \rho_\epsilon(z) - |z| \big| \le \epsilon$.
	\item If $\alpha_\epsilon$ denotes $d\rho_\epsilon/dz$, for all $z\in\mathbb R$ it holds that 
		\begin{equation}
			|\alpha(z)| \le 1 \quad\text{and}\quad \alpha_\epsilon(z)z \ge 0. \label{d.12}
		\end{equation}
	\item Let $\beta_\epsilon = d^2\rho_\epsilon/dz^2$. Then $\beta_\epsilon\ge0$ (due to the convexity of $\rho_\epsilon$).
\end{enumerate}
Such $\rho_\epsilon$ does exist; for example, if we define $\rho_\epsilon(z) = |z| - (1 - \frac2\pi)\epsilon$ if $|z|\ge\epsilon$, $\rho_\epsilon(z) = \frac{2\epsilon}\pi (1 - \cos\frac\pi{2\epsilon}|z|)$ if $|z|\le\epsilon$,
then this $\rho_\epsilon$ enjoys all of (a)--(d) above.

Since $\rho_\epsilon$ is differentiable, the functional $j_\epsilon$ is G\^ateaux differentiable, with its derivative $Dj_\epsilon(\eta)\in H^{1/2}(\Gamma_1)'$ computed by
\begin{equation*}
	\left< Dj_\epsilon(\eta), \xi \right>_{H^{1/2}(\Gamma_1)} = \int_{\Gamma_1}g\alpha_\epsilon(\eta)\xi\,ds \qquad (\forall \eta,\xi \in H^{1/2}(\Gamma_1)).
\end{equation*}

Now let us state the regularized problems.
\begin{Prob}{VI$_\sigma^\epsilon$-LBCF}
	For a.e.~$t\in(0,T)$, find $u_\epsilon(t)\in V_{\tau,\sigma}$  such that $u_\epsilon'(t) \in L^2(\Omega)^d$, $u_\epsilon(0) = u_0^\epsilon$ and 
	\begin{align*}
		&(u_\epsilon', v-u_\epsilon) + a_0(u_\epsilon, v-u_\epsilon) + a_1(u_\epsilon, u_\epsilon, v-u_\epsilon) + j_\epsilon(v_n) - j_\epsilon(u_{\epsilon n}) \\
		\ge &(f, v-u_\epsilon) \hspace{9cm}(\forall v\in V_{\tau,\sigma}).
	\end{align*}
\end{Prob}
\vspace{1mm}
\begin{Prob}{VE$_\sigma^\epsilon$-LBCF}
	For a.e.~$t\in(0,T)$, find $u_\epsilon(t)\in V_{\tau,\sigma}$  such that $u_\epsilon'(t) \in L^2(\Omega)^d$, $u_\epsilon(0) = u_0^\epsilon$ and 
	\begin{equation*}
		(u_\epsilon', v) + a_0(u_\epsilon, v) + a_1(u_\epsilon, u_\epsilon, v) + \int_{\Gamma_1}g\alpha_\epsilon(u_{\epsilon n}) v_n ds = (f, v) \qquad (\forall v\in V_{\tau,\sigma}).
	\end{equation*}
\end{Prob}
As in Proposition \ref{Prop c.1}, Problems $\mathrm{VI}_\sigma^\epsilon$ and $\mathrm{VE}_\sigma^\epsilon$ are equivalent.
The construction of the perturbed initial velocity $u_0^\epsilon$ is similar to that of SBCF.
In fact, since LBCF holds at $t=0$ by (L\ref{L3}), the Green formula leads to
\begin{align*}
	&a_0(u_0, v - u_0) + \int_{\Gamma_1} g(0) |v_n|\, ds - \int_{\Gamma_1} g(0) |u_{0n}|\, ds \\[-2mm]
	\ge &(-\nu\Delta u_0 + \nabla p_0, v - u_0) \hspace{5cm} (\forall v\in V_{\tau,\sigma}).
\end{align*}
We consider the regularized problem: find $u_0^\epsilon \in V_{\tau,\sigma}$ such that
\begin{align}
	&a_0(u_0^\epsilon, v - u_0^\epsilon) + \int_{\Gamma_1}\!\! g(0)\rho_\epsilon(v_n)\, ds - \int_{\Gamma_1}\!\! g(0)\rho_\epsilon(u_{0n}^\epsilon)\, ds \notag \\[-2mm]
	\ge &(-\nu\Delta u_0 + \nabla p_0, v - u_0^\epsilon) \hspace{5cm} (\forall v\in V_{\tau,\sigma}), \label{d.11}
\end{align}
which is equivalent to (cf. Proposition \ref{Prop c.1})
\begin{equation}
	a_0(u_0^\epsilon, v) + \int_{\Gamma_1} g(0)\alpha_\epsilon(u_{0n}^\epsilon)v_n\, ds = (-\nu\Delta u_0 + \nabla p_0, v) \qquad (\forall v\in V_{\tau,\sigma}). \label{d.17}
\end{equation}
The elliptic variational inequality (\ref{d.11}) admits a unique solution $u_0^\epsilon$, which is the perturbation of $u_0$ in question.
With this setting, we find:
\begin{Lem}\label{Lem d.1}
	{\rm (i)} When $\epsilon\to0$, $u_0^\epsilon\to u_0$ strongly in $H^1(\Omega)^d$.
	In particular, it follows that $u_0^\epsilon \to 0$ in $L^2(\Gamma_1)$.
	
	{\rm (ii)} $u_0^\epsilon \in H^2(\Omega)^d$ and
	\begin{equation}
		\|u_0^\epsilon\|_{H^2(\Omega)^d} \le C(\|\nu\Delta u_0 + \nabla p_0\|_{L^2(\Omega)^d} + \|g(0)\|_{H^1(\Gamma_1)}). \label{d.55}
	\end{equation}
\end{Lem}
\begin{proof}
	(i) is proved by the same way as Lemma \ref{Lem c.1}(i).
	Since $g(0) \in H^1(\Gamma_1)$ by (L\ref{L3}), (ii) is a direct consequence of \cite[Lemma 4.1]{S04}.
\end{proof}
\begin{Rem}
	By (i) and (L\ref{L4}), for sufficiently small $\epsilon>0$ we have
	\begin{equation}
		\|u_0^\epsilon\|_{L^2(\Omega)^d} \le 2\|u_0\|_{L^2(\Omega)^d},\quad \|u_0^\epsilon\|_{H^1(\Omega)^d} \le 2\|u_0\|_{H^1(\Omega)^d},\quad \|u_{0n}^\epsilon\|_{L^2(\Gamma_1)} \le \frac\alpha{8\gamma_1}, \label{d.20}
	\end{equation}
	where $\alpha$ and $\gamma_1$ are the constants in (\ref{b.3}) and (\ref{b.10}) respectively.
\end{Rem}
\begin{Rem} \label{Rem d.10}
	As in SBCF, if $\Gamma_0$ is $C^2$ and $\Gamma_1$ is $C^4$, then we can apply Lemma 4.1 of \cite{S04}.
	On the other hand, $g(0)\in H^{1/2}(\Gamma_1)$, stated in \cite{S04}, is actually insufficient to deduce the $H^2$-$H^1$ regularity (see the errata of \cite{S04}).
\end{Rem}

To solve Problem $\mathrm{VE}_\sigma^\epsilon$, let us construct approximate solutions by Galerkin's method.
Since $V_{\tau,\sigma}\subset H^1(\Omega)^d$ is separable, there exist $w_1,w_2,. . . \in V_{\tau,\sigma}$, linear independent to each other, such that $\bigcup_{m=1}^\infty \mathrm{span}\{w_k\}_{k=1}^m \subset V_{\tau,\sigma}$ dense in $H^1(\Omega)^d$.
Here we may assume $w_1 = u_0^\epsilon$.
\vspace{2mm}
\begin{Prob}{VE$_\sigma^{\epsilon,m}$-LBCF}
	Find $c_k\in C^2([0, T])\, (k=1, . . . ,m)$ such that $u_m\in V_{\tau,\sigma}$ defined by $u_m = \sum_{k=1}^m c_k(t)w_k$ satisfies $u_m(0) = u_0^\epsilon$ and
	\begin{align}
		&(u_m', w_k) + a_0(u_m, w_k) + a_1(u_m, u_m, w_k) + \int_{\Gamma_1} g\alpha_\epsilon(u_{mn})w_{kn}ds = (f, w_k) \notag \\[-2mm]
		&\hspace{9cm}	(k=1, . . . , m). \label{d.10}
	\end{align}
\end{Prob}
\noindent Since $\alpha_\epsilon\in C^1(\mathbb R)$, there exist unique solutions $c_k\in C^2([0, \tilde T])\, (k=1, . . . ,m)$ for some $\tilde T$, which may depend on $m$ and $\epsilon$ at this stage.

\begin{Prop} \label{Prop d.1}
	Assume {\rm (L\ref{L1})--(L\ref{L4})}, and let $\epsilon>0$ be sufficiently small so that $(\ref{d.20})$ holds.
	Then there exists some interval $(0, T')$ such that $u_m\in L^\infty(0, T'; V_{\tau,\sigma})$ and $u_m'\in L^\infty(0, T'; L^2(\Omega)^d) \cap L^2(0, T'; V_{\tau,\sigma})$ are uniformly bounded with respect to $m$ and $\epsilon$.
	Here, $T'$ is independent of $m$ and $\epsilon$.
\end{Prop}
\begin{proof}
	Due to space limitations, we sometimes simply write $\|u\|_{L^2}, \|g\|_{L^2}, . . . $, instead of $\|u\|_{L^2(\Omega)^2}, \|g\|_{L^2(\Gamma_1)}, . . . $, when there is no fear of confusion.
	
	First we consider the case $d=2$.
	Multiplying (\ref{d.10}) by $c_k(t)$ for $k=1, . . . , m$, adding them, using (\ref{b.3}), (\ref{b.10}) and (\ref{d.12}), we obtain
	\begin{equation}
		\frac12 \frac d{dt}\|u_m\|_{L^2}^2 + (\alpha - \gamma_1\|u_{mn}\|_{L^2(\Gamma_1)})\|u_m\|_{H^1}^2 \le (f, u_m). \label{d.13}
	\end{equation}
	Since $\|u_{mn}(t)\|_{L^2(\Gamma_1)}$ is continuous with respect to $t$ and (\ref{d.20}) holds, there exists a maximum value $T_1\in (0, \tilde T]$ of $t$ such that $\gamma_1 \|u_{mn}(t)\|_{L^2(\Gamma_1)} \le \frac\alpha4.$
	If this inequality holds for all $0\le t\le \tilde T$, we take $T_1 = \tilde T$.
	Noting $|(f, u_m)| \le \frac\alpha4 \|u_m\|_{H^1}^2 + \frac1\alpha \|f\|_{L^2}^2$, we find from (\ref{d.13}) that
	\begin{equation*}
		\frac d{dt}\|u_m\|_{L^2}^2 + \alpha\|u_m\|_{H^1}^2 \le C\|f\|_{L^2}^2 \qquad (0\le t\le T_1).
	\end{equation*}
	Hence $u_m \!\in\! L^\infty(0, T_1; L^2)\cap L^2(0, T_1; V_{\tau,\sigma})$ is bounded independently of $m$, $\epsilon$.
	
	Next, differentiating (\ref{d.10}), multiplying the resulting equation by $c_k'(t)$, and adding them, we obtain
	\begin{align}
		&(u_m'', u_m') + a_0(u_m', u_m') + a_1(u_m', u_m, u_m') + a_1(u_m, u_m', u_m') \notag \\
		&\hspace{1cm}	+ \int_{\Gamma_1}g' \alpha_\epsilon(u_{mn}) u_{mn}'\, ds +  \int_{\Gamma_1}g \beta_\epsilon(u_{mn}) |u_{mn}'|^2\, ds = (f', u_m'). \label{d.14}
	\end{align}
	Here, we estimate each term in (\ref{d.14}) as follows:
	{\allowdisplaybreaks
	\begin{align}
		|a_1(u_m', u_m, u_m')| &\le C\|u_m'\|_{L^2} \|u_m\|_{H^1} \|u_m'\|_{L^2} \notag \\
			&\le \frac\alpha{12} \|u_m'\|_{H^1}^2 + C\|u_m\|_{H^1}^2 \|u_m'\|_{L^2}, \label{d.15}
	\end{align}
	\vspace{-5mm}
	\begin{gather*}
		|a_1(u_m, u_m', u_m')| = \left| \int_{\Gamma_1}u_{mn} |u_m'|^2\,ds \right| \le \gamma_1\|u_{mn}\|_{L^2(\Gamma_1)} \|u_m'\|_{H^1}^2 \le \frac\alpha4 \|u_m'\|_{H^1}^2, \\
		\left|\int_{\Gamma_1}g' \alpha_\epsilon(u_{mn}) u_{mn}'\, ds\right| \le C\|g'\|_{L^2} \|u_m'\|_{H^1} \le \frac\alpha{12} \|u_m'\|_{H^1}^2 + C\|g'\|_{L^2}^2, \\
		\int_{\Gamma_1}g \beta_\epsilon(u_{mn}) |u_{mn}'|^2\, ds \ge 0, \\
		|(f', u_m')| \le \frac\alpha{12} \|u_m'\|_{H^1}^2 + C\|f'\|_{L^2}^2.
	\end{gather*}
	}Collecting these estimates, we derive from (\ref{d.14}) that for $0\le t\le T_1$
	\begin{equation}
		\frac d{dt}\|u_m'\|_{L^2} + \alpha\|u_m'\|_{H^1}^2 \le C(\|f'\|_{L^2}^2 + \|g'\|_{L^2}^2) + C\|u_m\|_{H^1}^2 \|u_m'\|_{L^2}^2. \label{d.16}
	\end{equation}
	Combining the technique used in Proposition \ref{Prop c.2} with (\ref{d.17}) and (\ref{d.55}), we observe that $\|u_m'\|_{L^\infty\!(0, T_1; L^2)}$, $\|u_m'\|_{L^2(0, T_1; V_{\tau,\sigma})}$, and $\|u_m\|_{L^\infty(0, T_1; V_{\tau,\sigma})}$ are bounded by $C(f,g,u_0,p_0)$.
	
	It remains to show that $T_1$ is bounded from below independently of $m,\,\epsilon$. 
	If $\gamma_1\|u_{mn}(T_1)\|_{L^2(\Gamma_1)} < \alpha/4$ and thus $T_1=\tilde T$, we can extend $u_m(t)$ beyond $t = \tilde T$ and repeat the above discussion until we reach either 
	\begin{equation*}
		\max_{0\le t\le T}\gamma_1\|u_{mn}(t)\|_{L^2(\Gamma_1)} \le \alpha/4 \quad\text{or}\quad \gamma_1\|u_{mn}(T_1)\|_{L^2(\Gamma_1)} = \alpha/4.
	\end{equation*}
	In the former case $T_1=T$. In the latter case, we have
	\begin{align*}
		\frac\alpha{8\gamma_1} &\le \|u_{mn}(T_1)\|_{L^2(\Gamma_1)} - \|u_{mn}(0)\|_{L^2(\Gamma_1)} \le \|u_{mn}(T_1) - u_{mn}(0)\|_{L^2(\Gamma_1)} \\
			&\le \int_0^{T_1} \|u_{mn}'(t)\|_{L^2(\Gamma_1)}dt \le C\int_0^{T_1} \|u_m'\|_{H^1(\Omega)^d}dt \le C\sqrt{T_1} \|u_m'\|_{L^2(0, T_1; V_{\tau,\sigma})}.
	\end{align*}
	Hence $T_1$ is bounded from below, and we complete the proof for $d=2$.
	
	Second let us consider the case $d=3$.
	What changes from $d=2$ is that (\ref{d.15}) is replaced with
	\begin{align*}
		|a_1(u_m', u_m, u_m')| &\le C \|u_m'\|_{L^2}^{1/2} \|u_m\|_{H^1} \|u_m'\|_{H^1}^{3/2} \\
			&\le \gamma_2\|u_m\|_{H^1} \|u_m'\|_{H^1}^2  \!+\!  C\|u_m\|_{H^1} \|u_m'\|_{L^2}^2,
	\end{align*}
	where $\gamma_2$ can be arbitrarily small.
	We choose $\gamma_2$ satisfying $\gamma_2\|u_0\|_{H^1} \le \frac\alpha{48}$, so that $\gamma_2\|u_0^\epsilon\|_{H^1} \le \frac\alpha{24}$ by virtue (\ref{d.20}).
	Let $T_2$ be the maximum value of $t\in(0, \tilde T]$ such that $\gamma_2\|u_m(t)\|_{H^1} \le \frac\alpha{12}$.
	If this inequality holds for all $t\in(0, \tilde T]$, we set $T_2 = \tilde T$.
	Such $T_2$ does exist, and if $T_2<\tilde T$ then $\gamma_2\|u_m(T_2)\|_{H^1} = \frac\alpha{12}$.
	
	Therefore, setting $T' = \min(T_1, T_2)$, instead of (\ref{d.16}) we get
	\begin{equation*}
		\frac d{dt}\|u_m'\|_{L^2} + \alpha\|u_m'\|_{H^1}^2  \le  C(\|f'\|_{L^2}^2 + \|g'\|_{L^2}^2) + C\|u_m\|_{H^1} \|u_m'\|_{L^2}^2 \quad (0\le t \le T').
	\end{equation*}
	As a consequence, we see that $\|u_m'\|_{L^2(0, T'; V_{\tau,\sigma})}$, $\|u_m'\|_{L^\infty(0, T'; L^2)}$, $\allowbreak \|u_m\|_{L^\infty(0, T'; V_{\tau,\sigma})}$ are bounded by $C(f,g,u_0,p_0)$.
	
	Now, if $T_1<\tilde T$ or $T_2<\tilde T$ then $T'$ are bounded from below as follows:
	\begin{align*}
		\frac\alpha{12\gamma_1} &\le \|u_{mn}(T')\|_{L^2(\Gamma_1)} - \|u_{mn}(0)\|_{L^2(\Gamma_1)} \le \int_0^{T'} \|u_{mn}'\|_{L^2(\Gamma_1)}dt \\
			&\le C\int_0^{T'} \|u_m'\|_{H^1}dt \le C\sqrt{T_1} \|u_m'\|_{L^2(0, T'; V_{\tau,\sigma})}, \\
		\frac\alpha{24\gamma_2} &\le \|u_m(T')\|_{H^1} - \|u_m(0)\|_{H^1} \le \int_0^{T'}\|u_m'\|_{H^1}dt \le \sqrt{T'} \|u_m'\|_{L^2(0, T'; V_{\tau,\sigma})}.
	\end{align*}
	When $T_1= \tilde T$ and $T_2 = \tilde T$, we can extend $u_m(t)$ beyond $t = \tilde T$ and repeat the above discussion.
	This completes the proof of Proposition \ref{Prop d.1}.
\end{proof}

The last step of the proof, namely, passing to the limits $m\to\infty$ and $\epsilon\to0$ can be carried out by the same way as Proposition \ref{Prop c.3}, with $n$ replaced by $\tau$ and vice versa.
This proves that a solution of Problem VI$_\sigma$ exists, which, combined with the uniqueness result, completes the proof of Theorem \ref{Thm d.1}.
\begin{Rem}
	At first glance one may think Theorem \ref{Thm d.1}, where we get only a time-local solution in spite of a smallness assumption on $u_0$ even if $d=2$, is too poor.
	However, in view of the fact that we obtain only time-local solutions in 2D case under the linear leak b.c.~(see \cite[Theorem 6]{HRT96} or \cite{Mar07}),
	such limitations cannot be avoided to some extent.
\end{Rem}
\begin{Rem}
	Under additional smallness assumptions on the data $f, g, u_0, p_0$, we can derive global existence results for both $d=2$ and $d=3$.
\end{Rem}

\section{Concluding Remarks}
By the discussion presented above, we have established the existence and uniqueness, while we did not get in touch with higher regularity such as
\begin{equation*}
	u\in L^\infty(0, T; H^2(\Omega)^d),\, p\in L^\infty(0, T; H^1(\Omega)).
\end{equation*}
This is because some regularity results for the elliptic cases are not available.
For instance, Problem VI$_{\sigma}$-SBCF is rewritten as
\begin{align*}
	a_0(u, v - u) + j(v_\tau) - j(u_\tau) &\ge (f, v - u) - (u', v - u) - a_1(u, u, v - u) \\
		&=: \left<F(t), v - u\right>_{V_{n,\sigma}} \qquad (\forall v\in V_{n,\sigma}),
\end{align*}
with $F(t)\in L^p(\Omega)^d$ for some $p<2$.
If we prove this elliptic variational inequality has a unique solution in $W^{2,p}(\Omega)^d$ when $p<2$, then a technique similar to \cite[Theorems III.3.6 and III.3.8]{Tem77} allows us to deduce $u(t)\in H^2(\Omega)^d$.
Thereby, we need to extend the regularity theory of \cite{S04} to cases $p\neq2$.

\section*{Acknowledgements}
The author would like to thank for Professor Norikazu Saito for encouraging him through valuable discussions.
This work was supported by Grant-in-Aid for JSPS Fellows and CREST, JST.

\providecommand{\bysame}{\leavevmode\hbox to3em{\hrulefill}\thinspace}
\providecommand{\MR}{\relax\ifhmode\unskip\space\fi MR }
\providecommand{\MRhref}[2]{%
  \href{http://www.ams.org/mathscinet-getitem?mr=#1}{#2}
}
\providecommand{\href}[2]{#2}

\end{document}